\documentclass{amsart}
\usepackage[utf8]{inputenc}
\usepackage{enumerate}
\usepackage{setspace}
\usepackage{geometry}
\usepackage{graphicx}
\usepackage{subfigure}
\usepackage{amsthm,amsmath,amssymb}
\usepackage{caption}
\usepackage{graphicx, subfig}
\usepackage{float}
\usepackage{booktabs}
\usepackage[ruled]{algorithm2e}

\usepackage{fancyhdr}
\usepackage{mathrsfs}

\usepackage{framed}

\usepackage{caption}
\usepackage{subfigure}
\usepackage{indentfirst}
\newtheorem{theorem}{Theorem}[section]
\newtheorem{definition}{Definition}[section]
\newtheorem{corollary}{Corollary}[section]
\newtheorem{remark}{Remark}[section]
\numberwithin{equation}{section}
\begin{document}

\title{Infinite-thin shock layer solutions for stationary compressible conical flows and numerical results via Fourier spectral method}

\author{Aifang Qu}
\author{Xueying Su}
\author{Hairong Yuan}

\address[A. Qu]{Department of Mathematics, Shanghai Normal University,
Shanghai,  200234,  China}
\email{\tt afqu@shnu.edu.cn}

\address[X. Su]{Center for Partial Differential Equations, School of Mathematical Sciences,
East China Normal University, Shanghai
200241, China}
\email{\tt suxueying789@163.com}

\address[H. Yuan]{School of Mathematical Sciences and Shanghai Key Laboratory of Pure Mathematics and Mathematical Practice,
East China Normal University, Shanghai
200241, China}
\email{\tt hryuan@math.ecnu.edu.cn}

\maketitle
\allowbreak
\begin{abstract}
    We  consider the problem of uniform steady supersonic Euler flows passing a straight conical body with attack angles, and study Radon measure solutions describing the infinite-thin shock layers, particularly for the Chaplygin gas and limiting hypersonic flows. As a byproduct, we obtain the generalized Newton-Busemann pressure laws.  To construct the Radon measure solutions containing weighted Dirac measures supported on the edge of the cone on the 2-sphere, we  derive some highly singular and non-linear ordinary differential equations (ODE). A numerical algorithm based on the combination of Fourier spectral method and Newton's method is developed to solve the physically desired nonnegative and periodic solutions of the ODE. The numerical simulations for different attack angles exhibit proper theoretical properties and excellent accuracy, thus would be useful for engineering of hypersonic aerodynamics.
\end{abstract}

\newcommand\keywords[1]{\textbf{Keywords}: #1}

\section{Introduction}
Supersonic conical flow is an important prototypical problem in mathematical gas dynamics  due to its wide applications and tremendous challenges. It is observed that if the Mach number of the upcoming flow is large, shock layer (i.e., the region between the shock-front and the cone) will be extremely thin, and density would blow up to infinity. Qu and Yuan \cite{1} studied supersonic flow of polytropic gas passing a straight cone. For the particular case without attack angle, they constructed a Radon measure solution to the hypersonic-limit problem, with density containing a Dirac measure supported on the surface of the cone, and then proved the celebrated Newton sine-squared law of hypersonic aerodynamics. For the general case, it remains open even for finding a numerical solution, since the resultant ODE are quite singular and highly nonlinear, while one needs to establish nonnegative periodic solutions.
In this work, we consider infinite-thin shock layer solutions to steady compressible flow passing the cone with attack angles, especially for Chaplygin gas, and then propose a numerical method to solve the derived ODE via Fourier spectral method. 
For more background and introductions on conical flows, we refer to \cite{2,3,4,6,7} and references therein.

In the rest of this section, we explain the notations, formulate the problem of supersonic conical flows, and present the concept of Radon measure solution to the problem. Section \ref{sec2} is devoted to deriving the ODE governing the weights of the Dirac measures, which describe the strength of the infinite-thin shock layers. In Section \ref{sec3} we consider the special case of Chaplygin gas. In the final Section \ref{sec4}, we exhibit the Fourier spectral method of solving the nonnegative periodic solutions and present many numerical results demonstrating its efficiency.

\subsection{Equations of conical flow}In this paper, we adopt the sphere coordinates as used in \cite{1}, see Figure \ref{fig1}. The sphere coordinates of the standard Euclidean space $\mathbb{R}^3$ $(x=(x^1,x^2,x^3))$ is $(\theta,\phi,r)$, given by
\begin{equation*}
    x^1=r\cos\theta,\quad x^2=r\sin\theta\cos\phi,\quad x^3=r\sin\theta\sin\phi,\quad  \theta\in[0,\pi],\quad \phi\in[-\pi,\pi].
\end{equation*}
The natural frame under sphere coordinates is denoted as $(\Vec{\partial_\theta},\Vec{\partial_\phi},\Vec{\partial_r}):$
\begin{equation*}
    \Vec{\partial_\theta}=\frac{\partial x}{\partial \theta},\quad \Vec{\partial_\phi}=\frac{\partial x}{\partial \phi},\quad \Vec{\partial_r}=\frac{\partial x}{\partial r}.
\end{equation*}
The velocity field $V\in\mathbb{R}^3$ of the flow is written as $V=u+w\vec{\partial_r}=u^\theta\vec{\partial_\theta}+u^\phi\vec{\partial_\phi}+w\vec{\partial_r}$, with $u^\theta,~u^\phi,~w$ the component along $\Vec{\partial_\theta},~\Vec{\partial_\phi}$ and  $\Vec{\partial_r}$ respectively. Thus, the tangential component of the velocity is $u=u^\theta\vec{\partial_\theta}+u^\phi\vec{\partial_\phi}$, and $w\partial_r$ is the radial component. We set $\rho$ be the density of mass of the gas, $E$ the total ethalpy per unit mass of gas, and $p$ the pressure. We also record here  the divergence and gradient operator in $\mathbb{R}^3$ and on the unit sphere $S^2\subset\mathbb{R}^3$ under sphere coordinates as follows:
\begin{align*}
    & \mathrm{Div}\,V=\frac{1}{\sqrt{G}}\Big(\partial_\theta(\sqrt{G} u^\theta)+\partial_\phi(\sqrt{G} u^\phi)+\partial_r(\sqrt{G} w)\Big),\quad \sqrt{G}\doteq r^2\sin\theta,\\
    & \mathrm{div}\,u=\frac{1}{\sqrt{g}}\Big(\partial_\theta(\sqrt{g}u^\theta)+\partial_\phi(\sqrt{g}u^\phi\Big),\quad \sqrt{g}\doteq\sin\theta,\\
    & \mathrm{Grad}\,p=\frac{1}{r^2}\partial_\theta p\cdot\Vec{\partial_\theta}+\frac{\sin^2\theta}{r^2}\partial_\phi p \cdot\Vec{\partial_\phi}+\partial_rp\cdot\Vec{\partial_r},\quad \mathrm{grad}\,p= \partial_\theta p\cdot\Vec{\partial_\theta}+\sin^2\theta\partial_\phi p \cdot\Vec{\partial_\phi}.
\end{align*}
\begin{figure}[htb]
\centering
\includegraphics[scale=0.45]{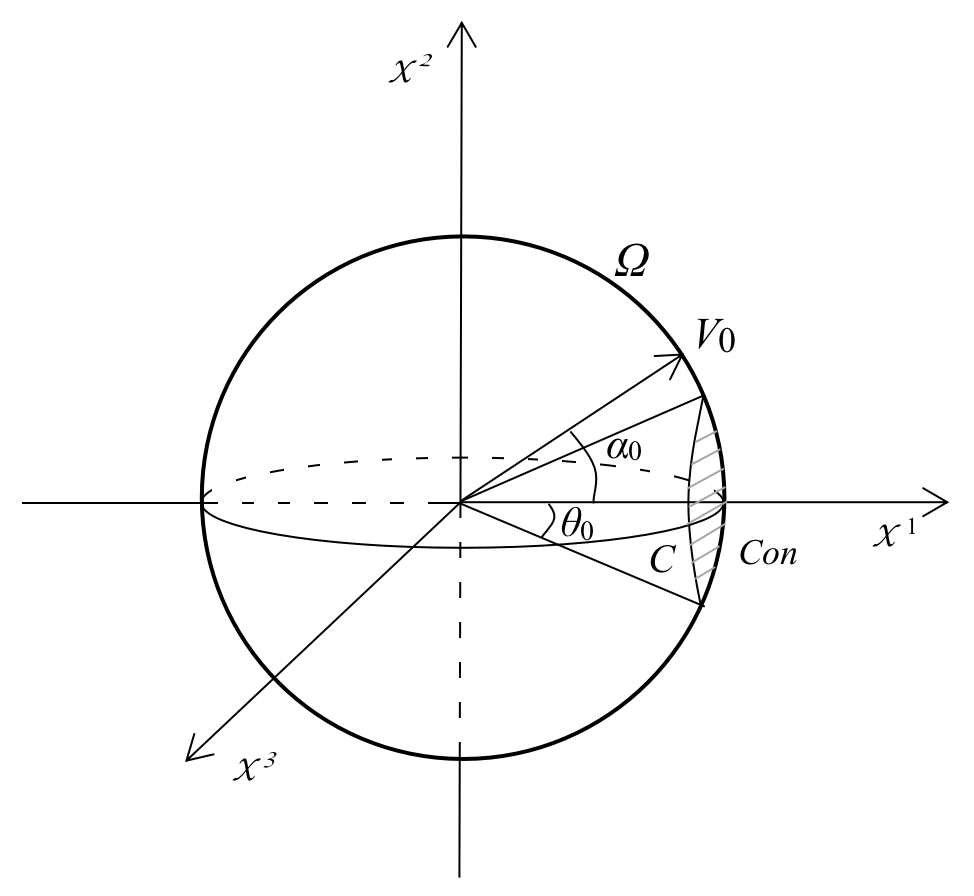}
\caption{Uniform supersonic flow with velocity $V_0$ past a cone $C$ with semi-vertex angle $\theta_0$ and attack angle $\alpha_0$.}\label{fig1}
\end{figure}

Suppose that there is an infinite straight cone $\mathfrak{C}$ in $\mathbb{R}^3$ whose vertex is located at the original point $O$. The cone is symmetric with respect to the $x^1$-axis and its semi-vertex angle is $\theta_0$. Uniform supersonic gas from the left half-space flow towards the cone with attack angle $\alpha_0$, and the flow outside the cone is assumed to be governed by
the non-isentropic compressible Euler system
\begin{align}
    & \mathrm{Div}(\rho V)=0, \label{1.1}\\
    & \mathrm{Div}(\rho V\otimes V)+\mathrm{Grad}\, p=0, \label{1.2} \\
    & \mathrm{Div}(\rho EV)=0, \label{1.3}
\end{align}
which represents the conservation of mass, momentum and energy respectively. State equation of the gas is given by
\begin{equation}
    p=p(\rho, E). \label{1.4}
\end{equation}
After dimensionless scalings (cf. \cite{1}), we may assume that the upcoming flow is
\begin{equation}
    U_0=(\rho_0=1,V_0\in S^2, E_0),\quad E_0\geq\tfrac{1}{2},\label{1.5}
\end{equation}
and \eqref{1.4} becomes
\begin{equation}
    p_0=\Tilde{p}(\rho_0, E_0). \label{1.6}
\end{equation}
On the surface of the cone, we prescribe the slip condition
\begin{equation}
     (V,N)=0, \label{1.7}
\end{equation}
with $N$ the unit outward normal vector on the surface of the cone.

Given that the above problem is invariant under scaling $x\mapsto\alpha x,~\forall\alpha>0$, we may consider that the flow is defined on $\Omega\doteq S^2/Con$, where $Con$ is the common part of the cone and $S^2$. Let $C\doteq\partial\Omega$ be the edge of $\Omega$, and $\mathbf{n}$ the unit outward normal vector of $\Omega$ along $C$.
By direct deduction with differential geometry as in \cite{1}, which we omit here because of the similarity, one has the following compressible Euler system of conical flows on $\Omega$
\begin{align}
    & \mathrm{div}(\rho u)+2\rho w=0,\label{1.8}\\
    & \mathrm{div}(\rho u E)+2\rho wE=0,\label{1.9}\\
    & \mathrm{div}(\rho wu)+2\rho w^2-\rho|u|^2=0,\label{1.10}\\
    & \mathrm{div}(\rho u\otimes u)+3\rho wu+\mathrm{grad}\, p=0,\label{1.11}
\end{align}
which represents the conservation of mass, energy, radial momentum, and tangential momentum respectively.
Slip condition \eqref{1.6} then reduces to
\begin{equation}
    (u,\mathbf{n})=0 \quad    \text{on} \quad C.\label{1.12}
\end{equation}

Our goal is to study the following problem.
\begin{framed}
    \textbf{Problem A}: find a solution to \eqref{1.8}-\eqref{1.11}, \eqref{1.5}, \eqref{1.6}, and \eqref{1.12}.
\end{framed}

\subsection{Radon measure solutions} When Mach number of the upcoming flow is sufficiently large, it is observed that shock layer will be extremely thin and mass will concentrate on the surface of the cone, cf. \cite{4,7}. Hence we consider a solution to Problem A in the class of Radon measures. Let $m$ be a Radon measure on $S^2$.~The pairing between $m$ and a test function $\psi(\theta,\phi)\in C(S^2)$, the set of continuous functions on $S^2$, is given by
$$\langle m,~\psi\rangle=\int_{S^2}\psi(\theta,\phi)~\mathrm{d}m(\theta,\phi).$$
For example, let $C(s)$ be a Lipschitz curve on $S^2$, with arc-length parameter $s\in [0,L]$. ~Then $W(s)\delta_R$, the Dirac measure supported on $C$ with weight $W(s)\in L^1([0,L])$, is defined by
$$\langle W(s)\delta_C,~\psi\rangle=\int_0^L W(s)~\psi\Big|_C~\mathrm{d}s.$$
We also denote $\mathcal{H}^2$ as the standard Hausdorff measure on $S^2$.
\begin{definition}
    Let $m_a,~ n_a,~m_e,~n_e,~m_r,~ n_r,~ m_t,~ n_t,~\varrho,~\wp$ be Radon measures on $\overline{\Omega}$, and $W_C(s)$  a positive integrable function for $s\in[0,L]$. Suppose that

(\romannumeral1) for any  $ \psi\in C^1(S^2)$ and $C^1$ vector field $v$, there hold
\begin{align}
    & \langle m_a,\nabla\psi\rangle=2\langle n_a,\nabla\psi\rangle,\label{1.13}\\
    & \langle m_e,\nabla\psi\rangle=2\langle n_e,\nabla\psi\rangle,\label{1.14}\\
    & \langle m_r,\nabla\psi\rangle=2\langle n_r,\nabla\psi\rangle-\langle n_t,\nabla\psi\rangle, \label{1.15}\\
    & \langle m_t,\mathrm{D}v\rangle+\langle \wp,\mathrm{div} v\rangle=3\langle m_r,v\rangle+\langle W_C\mathbf{n}\delta_C,v\rangle; \label{1.16}
\end{align}

(\romannumeral2) $\varrho,~\wp$ are non-negative measures, $m_a,~m_r,~m_e,~n_e,~m_t,~n_a,~n_r,~n_t,~\wp$  are absolute-continuous with respect to $\varrho$, and there also exist  $\varrho$-a.e. functions $w,~E$ and vector field $u$ such that the Radon-Nikodym derivatives satisfy
\begin{align}
    & w=\frac{\mathrm{d}m_r/\mathrm{d}\varrho}{\mathrm{d}m_a/\mathrm{d}\varrho}=\frac{\mathrm{d}n_r/\mathrm{d}\varrho}{\mathrm{d}n_a/\mathrm{d}\varrho}=\frac{\mathrm{d}n_a}{\mathrm{d}\varrho},\label{1.17}\\
&E=\frac{\mathrm{d}m_e/\mathrm{d}\varrho}{\mathrm{d}m_a/\mathrm{d}\varrho}=\frac{\mathrm{d}n_e/\mathrm{d}\varrho}{\mathrm{d}n_a/\mathrm{d}\varrho},\label{1.18}\\
    & u=\frac{\mathrm{d}m_a}{\mathrm{d}\varrho},~~|u|^2=\frac{\mathrm{d}n_t}{\mathrm{d}\varrho},~~u\otimes u=\frac{\mathrm{d}m_t}{\mathrm{d}\varrho};\label{1.19}
\end{align}

(\romannumeral3) if $\varrho, \wp\ll\mathcal{H}^2$, and their Radon-Nikodym derivatives are
\begin{equation}
    \rho=\frac{\mathrm{d}\varrho}{\mathrm{d}\mathcal{H}^2},\quad p=\frac{\mathrm{d}\wp}{\mathrm{d}\mathcal{H}^2},
\end{equation}
then \eqref{1.4} holds $\mathcal{H}^2$-a.e., and classical entropy condition is valid for discontinuities of functions $\rho, u, w, E$ in this case.

Then we call $(\varrho, u, w, E)$ a Radon measure solution to Problem A.
\end{definition}

We remark that $W_c\mathbf{n}$ represents the (scaled) force of lift/drag acting by the flow on the conical body. From mathematical point of view, the above definition works for gas with general state function \eqref{1.4}, not necessarily restricted to the special Chaplygin gas or limiting hypersonic flows (pressureless Euler flows).

\section{Radon measure solutions of infinite-thin shock layers}\label{sec2}
We noticed that in the hypersonic aerodynamics, it is known that if the Mach number of the upcoming flow is large enough, the shock layer would be quite thin and mass concentrate on the surface of the cone.
We wish to construct Radon measure solutions with such structure, i.e., assuming that a Radon measure solution to Problem A is given by
\begin{align}
    & m_a=\rho_0u_0\mathbb{I}_\Omega \mathrm{d}\mathcal{H}^2+W_a(s)\delta_C,~ ~~~~~~~~~~~ n_a=\rho_0w_0\mathbb{I}_\Omega \mathrm{d}\mathcal{H}^2+w_a(s)\delta_C, \label{2.1}\\
    & m_e=\rho_0u_0E_0\mathbb{I}_\Omega \mathrm{d}\mathcal{H}^2+W_e(s)\delta_C,~ ~~~~~~~~~~~ n_e=\rho_0w_0E_0\mathbb{I}_\Omega \mathrm{d}\mathcal{H}^2+w_e(s)\delta_C, \label{2.2}\\
    & m_r=\rho_0u_0w_0\mathbb{I}_\Omega \mathrm{d}\mathcal{H}^2+W_r(s)\delta_C,~ ~~~~~~~~ n_r=\rho_0w_0^2\mathbb{I}_\Omega \mathrm{d}\mathcal{H}^2+w_e(s)\delta_C,\label{2.3}\\
    & m_t=\rho_0u_0\otimes\rho_0\mathbb{I}_\Omega \mathrm{d}\mathcal{H}^2+W_t(s)\delta_C,~ ~~~~~~ n_t=\rho_0|u_0|^2\mathbb{I}_\Omega \mathrm{d}\mathcal{H}^2+w_t(s)\delta_C,\label{2.4}\\
    & \varrho=\rho_0\mathbb{I}_\Omega\mathrm{d}\mathcal{H}^2+w_\rho\delta_C,~ ~~~~~~~~~~~~~~~~~~~~~~ \wp=\Tilde{p}(\rho_0, E_0)\mathrm{d}\mathcal{H}^2.\label{2.5}
\end{align}
In the above, we have set $\mathbb
I_\Omega(\theta,\phi)$ to be the indicator function of $\Omega$, i.e., $\mathbb
I_\Omega=1$ if $(\theta,\phi)\in\Omega$,  and  $\mathbb
I_\Omega=0$ otherwise. However, observing that we could always construct Radon measure solutions with such structure mathematically, even if the Mach number of the upcoming supersonic flow is not large. (The uniqueness of solutions is a quite delicate issue that we would not touch in this work.)

In the following, we denote $\mathbf{t}(s)$ and $\mathbf{n}(s)$ as the unit tangential and normal vector of $\Omega$ along $C$ respectively.
As in \cite{1}, substituting \eqref{2.1} into \eqref{1.13}, and supposing
\begin{equation*}
    W_a(s)=W_a^n(s)\mathbf{n}(s)+W_a^t(s)\mathbf{t}(s),
\end{equation*}
one has
\begin{align}
    & W_a^n(s)=0,\label{2.6}\\
    & (W_a^t(s))'+2w_a(s)=\rho_0(u,\mathbf{n}).\label{2.7}
\end{align}
Similarly, substituting \eqref{2.2} into \eqref{1.4}, one comes to
\begin{align}
    & W_e^n(s)=0,\label{2.8}\\
    & (W_e^t(s))'+2w_e(s)=\rho_0E_0(u,\mathbf{n}),\label{2.9}
\end{align}
where
\begin{equation*}
    W_e(s)=W_e^n(s)\mathbf{n}(s)+W_e^t(s)\mathbf{t}(s).
\end{equation*}
Substituting \eqref{2.3} into \eqref{1.15}, while by denoting
\begin{equation*}
    W_r(s)=W_r^n(s)\mathbf{n}(s)+W_r^t(s)\mathbf{t}(s),
\end{equation*}
one has
\begin{align}
    & W_r^n(s)=0,\label{2.10}\\
    & (W_r^t(s))'+2w_r(s)-w_t(s)=\rho_0w_0(u,\mathbf{n}).\label{2.11}
\end{align}
Substituting \eqref{2.4} and \eqref{2.5} into \eqref{1.16}, and with the decomposition
\begin{equation*}
    W_t = W_t^{nn}\mathbf{n}\otimes\mathbf{n} + W_t^{nt}\mathbf{n}\otimes \mathbf{t} + W_t^{tn}\mathbf{t}\otimes \mathbf{n} + W_t^{tt}\mathbf{t}\otimes \mathbf{t},
\end{equation*}
one obtains
\begin{align}
   &  W_t^{nn}=W_t^{nt}=W_t^{tn}=0,\label{2.12}\\
   & W_C = (W_t,\mathrm{D}n) + \rho_0(u_0, \mathbf{n})^2+\Tilde{p}(\rho_0,E_0), \label{2.13}\\
   & -(W_t,\mathrm{D}\mathrm{t})+(W_t^{tt})+3W_r^t = \rho_0(u_0, \mathbf{n})(u_0,\mathbf{t}). \label{2.14}
\end{align}

 By Definition 1.1, the unknowns $w_\rho(s),~w(s),~u(s)$ are determined by
 \begin{align}
    & u= W_a(s)/w_\rho(s), \quad |u(s)|^2 = w_t(s)/w_\rho(s),\label{2.15}\\
    & u(s)\otimes u(s) = W_t(s)/w_\rho(s),\label{2.16}\\
    & w = W_r(s)/.W_a(s) = w_r(s)/w_a(s) = w_a(s)/w_\rho(s).\label{2.17}
 \end{align}
 Here, /. means $W_r(s)$ and $W_a(s)$ are linearly dependent.
 From \eqref{2.15}, and writing
\begin{equation*}
u(s)=u^n(s)\mathbf{n}+u^t(s)\mathbf{t},
\end{equation*}
one has
 \begin{equation}
    u^n(s)=0,\quad u^t(s)=W_a^t(s)/w_\rho(s).
 \end{equation}
From \eqref{2.16}, one gets
 \begin{equation}
    u^t(s)^2=w_t(s)/w_\rho(s),
 \end{equation}
and therefore with \eqref{2.12},
 \begin{equation}
    W_t^{tt}(s)=w_\rho(u^t)^2.
 \end{equation}
Thus, \eqref{2.17} indicates that
 \begin{equation}
    W_r^t = w_\rho u^tw,\quad w_a =w_\rho w,\quad w_r = w_\rho w^2.
 \end{equation}
 Therefore from \eqref{2.14}, \eqref{2.9} and \eqref{2.11}, we have the following equations
\begin{align}
    & (w_\rho(u^t)^2)'+3w_\rho u^tw=\rho_0(u_0,~\mathbf{n})(u_0,~\mathrm{t})\triangleq a(s), \label{2.22}\\
    & (w_\rho u^tw)'-w_\rho(u^t)^2+2w_\rho w^2=\rho_0w_0(u_0,~\mathbf{n})\triangleq b(s),\label{2.23}\\
    & (w_\rho u^t)'+2w_\rho w=\rho_0(u_0,\mathbf{n}), \label{2.24}
\end{align}
and
\begin{equation}
    W_C=w_\rho(u^t)^2(\mathbf{t},\mathrm{D}_\mathbf{t}\mathbf{n}) + \rho_0(u_0, \mathbf{n})^2+\Tilde{p}(\rho_0, E_0).
\end{equation}
Hence we have the following theorem on conical flows with infinite-thin shock layers.
\begin{theorem}
    Suppose that $w_\rho(s),u^t(s),w(s)$ are solutions to equations \eqref{2.22}-\eqref{2.24}. Then Problem A admits a Radon measure solution given by
    \begin{align*}
&\varrho=\rho_0\mathbb{I}_\Omega\mathrm{d}\mathcal{H}^2+w_\rho(s)\delta_C,\quad u(s)=u_0\mathbb{I}_\Omega+u^t(s)\mathbf{t}\mathbb{I}_C,\\ &w(s)=w_0\mathbb{I}_\Omega+w(s)\mathbb{I}_C,\quad E=E_0\mathbb{I}_\Omega+E_0\mathbb{I}_C,
    \end{align*}
    provided that
    \begin{equation*}
    W_C=w_\rho(u^t)^2(\mathbf{t},\mathrm{D}_\mathbf{t}\mathbf{n}) + \rho_0(u_0, \mathbf{n})^2+\Tilde{p}(\rho_0, E_0)>0.
    \end{equation*}
\end{theorem}

We now derive a second order ODE from \eqref{2.22}-\eqref{2.24}. As in \cite{1}, by setting
\begin{equation}
    f(s)\doteq w_\rho (u^t)^2,\quad h(s)\doteq w_\rho u^tw,\quad y(s)\doteq w_\rho u^t,\label{1}
\end{equation}
\eqref{2.22}-\eqref{2.24} read
\begin{align}
    & f'+3h=a(s),\label{2.26}\\
    & h'-f+2h^2/f=b(s),\label{2.27}\\
    & y'+2hy/f=\rho_0(u_0,\mathbf{n}).\label{2.28}
\end{align}
We notice that \eqref{2.28} is already decoupled, while \eqref{2.26} implies that
\begin{equation}
    h(s)=(a(s)-f'(s))/3. \label{2.29}
\end{equation}
Then \eqref{2.27} becomes
\begin{equation}
    ff''-\frac{2}{3}f'^2+\frac{4}{3}a(s)f'+3f^2+(3b(s)-a'(s))f=\frac{2}{3}a(s)^2.
\end{equation}

Changing to $\phi$-variable, $\phi\in[-\pi,\pi]$, with $\mathrm{d}s/\mathrm{d}\phi=\sin\theta_0$, and by the validity of (also see \cite{1})
\begin{align*}
    & \rho_0=1,\\
    & (u_0,~\mathbf{n})=\cos\alpha_0 \sin \theta_0-\sin\alpha_0 \cos \theta_0 \cos\phi,\\
    & (u_0,~\mathbf{t})=\sin\alpha_0 \sin\phi,\\
    & w_0=\cos\alpha_0 \cos \theta_0+\sin\alpha_0 \sin \theta_0\cos\phi,
\end{align*}
we finally come to the ODE
\begin{framed}
    \begin{equation}
    \begin{aligned}
        f\ddot{f}-\frac{2}{3}\dot{f}^2+ (a_1\sin\phi+ & a_2\sin2\phi)\dot{f}+a_3f^2\\
        &+(a_4+a_5\cos\phi+a_6\cos2\phi)f=\frac{3}{8}(a_1\sin\phi+a_2\sin2\phi)^2. \label{2.31}
    \end{aligned}
\end{equation}
\end{framed}
In the equation, $\dot{f}$ represents $\mathrm{d}f(\phi)/\mathrm{d}\phi$, and
\begin{align*}
    & a_1=-\frac{2}{3}\sin^2\theta_0\sin(2\alpha_0),\quad a_2=\frac{1}{3}\sin(2\theta_0)\sin^2{\alpha_0},~\\
    & a_3=3\sin^2\theta_0,\quad a_4=\frac{3}{4}\sin^2\theta_0\sin(2\theta_0)(3\cos^2\alpha_0-1),~\\
    & a_5=\frac{1}{2}\sin^2\theta_0\sin(2\alpha_0)(1-3\cos(2\theta_0)),\quad a_6=-\sin\theta_0\cos\theta_0\sin^2\alpha_0(1+\frac{2}{3}\sin^2\theta_0).
\end{align*}
Besides, if $f(\phi)$ is a solution to equation \eqref{2.31}, then the lift/drag on the cone is given by
\begin{framed}
    \begin{equation}
    W_C(\phi)=(\cos\alpha_0\sin\theta_0-\sin\alpha_0\cos\theta_0\cos\phi)^2-f(\phi)\cot\theta_0+\Tilde{p}(\rho_0, E_0),\quad \phi\in[-\pi,\pi].\label{2.32}
\end{equation}
\end{framed}

\begin{remark}
    We call \eqref{2.32} as the generalized Newton-Busemann pressure law for conical flow with attack angle $\alpha_0$. We require that $W_C(\phi)>0$ to guarantee that the mass indeed concentrates on the surface of the cone.
\end{remark}

\begin{remark}
   Observing that $f$ stands for the double of the tangential kinetic energy of the concentrated gas. Due to physical considerations and symmetry with respect to the  $x^1$-axis, $f$ should be a non-negative even periodic $C^2$ function, with period $2\pi$, and $f(-\pi)=f(0)=f(\pi)=0,~\dot{f}(-\pi)=\dot{f}(0)=\dot{f}(\pi)=0$.
\end{remark}

\begin{remark}
When $\alpha_0\neq0$, finding an analytical solution to \eqref{2.31} is an open problem. It is also not straightforward to construct a numerical periodic solution, since it is non-autonomous, with non-homogeneous terms, and singular at $\phi=0, \pm\pi$, as $f(\phi)$ takes value $0$ when $\phi=0$ and $\pm\pi$. Therefore, we will propose a numerical method to solve the ODE in Section \ref{sec4} (cf. \cite{8,5}).
\end{remark}

\section{Chaplygin conical flows}\label{sec3}
We now consider the case when the upcoming flow is Chaplygin gas, with state function
\begin{equation}
    p=-\frac{A}{\rho},\quad A>0\label{3.1}
\end{equation}
and assuming $E=E_0$ being a given constant in the whole flow field.
Thus,
\begin{equation}
    p_0=\Tilde{p}(\rho_0, E_0)=-\frac{1}{\rho_0M_\infty^2},\label{3.2}
\end{equation}
with $M_\infty$ the Mach number of the upcoming flow. Recall that $\rho_0=1$ by the non-dimensional scalings.

\subsection{Conical Chaplygin flow without attack angle} We first deal with the particular case that $\alpha_0=0$. By Theorem 2.1, one reaches the following corollary.
\begin{corollary}
    When Mach number of the upcoming Chaplygin gas $M_\infty>\displaystyle{\frac{1}{\sin\theta_0}},$ Problem A (for Chaplygin gas) has a Radon measure solution
    \begin{equation*}
    \varrho=\rho_0\mathbb{I}_\Omega\mathrm{d}\mathcal{H}^2+\frac{1}{2}\tan\theta_0\delta_C,\quad w=\cos\theta_0\mathbb{I}_\Omega+\cos\theta_0\mathbb{I}_C,\quad
    u=u_0\mathbb{I}_\Omega,\quad E=E_0\mathbb{I}_\Omega+E_0\mathbb{I}_C,
    \end{equation*}
    and the pressure on C is given by
    \begin{equation*}
    W_C=\sin^2\theta_0-\frac{1}{M_\infty^2}.
    \end{equation*}
\end{corollary}
    \begin{proof}
    When $\alpha_0=0$, by \eqref{2.22}-\eqref{2.24}, one obtains the following equations
    \begin{align*}
        & (w_\rho(u^t)^2)'+3w_\rho u^tw=0,\\
    & (w_\rho u^tw)'-w_\rho(u^t)^2+2w_\rho w^2=\cos\theta_0\sin\theta_0,\\
    & (w_\rho u^t)'+2w_\rho w=\sin\theta_0.
    \end{align*}
    They yield the solution
    \begin{equation*}
        w_\rho(s)=\frac{1}{2}\tan\theta_0,\quad u^t(s)=0, \quad w(s)=\cos\theta_0.
    \end{equation*}
    The requirement
    \begin{equation*}
    W_C=\sin^2\theta_0-\frac{1}{\rho_0M_\infty^2}>0
    \end{equation*}
    justifies the assumption
    \begin{equation*}
    M_\infty >\sqrt{\frac{1}{\rho_0\sin^2\theta_0}}=\frac{1}{\sin\theta_0}.
    \end{equation*}
    \end{proof}
For limiting hypersonic flow, namely $M_{\infty}=\infty$, we have $W_C=\sin^2\theta_0$, which is the classical Newton sine-squared law for cones (cf. \cite[Sections 12.9 and 15.4]{4}).
\subsection{General case with attack angle}
We have the following corollary directly from Theorem 2.1.
\begin{corollary}
    Let $w_\rho(s),u^t(s),w(s)$ be the solutions to equations \eqref{2.22}-\eqref{2.24}, and there also holds
    \begin{equation*}
        M_\infty>\sup_{\phi\in[-\pi,\pi]}\Big(\frac{1}{(\cos\alpha_0\sin\theta_0-\sin\alpha_0\cos\theta_0\cos\phi)^2-w_\rho (u^t)^2(s(\phi))\cot\theta_0}\Big)^{1/2},
    \end{equation*}
    with $M_\infty$ the Mach number of the upcoming Chaplygin flow.
    Then Problem A admits a measure solution given by
    \begin{align*}
    & \varrho=\rho_0\mathbb{I}_\Omega\mathrm{d}\mathcal{H}^2+w_\rho(s)\delta_C,\quad u(s)=u_0\mathbb{I}_\Omega+u^t(s)\mathbf{t}\mathbb{I}_C,\\
    & w(s)=w_0\mathbb{I}_\Omega+w(s)\mathbb{I}_C,\quad E=E_0\mathbb{I}_\Omega+E_0\mathbb{I}_C.
    \end{align*}
\end{corollary}

\section{Numerical results with Fourier spectral method}\label{sec4}
\subsection{Non-linear system and spectral method}
Our goal is to find a $C^2$ periodic solution to equation \eqref{2.31}, which should be an even nonnegative function, satisfying $f(0)=f(\pi)=0, ~\dot{f}(0)=\dot{f}(\pi)=0$ owing to physical considerations.
Since the solution is an even periodic function, we use here the Fourier series of $f(\phi)$:
\begin{equation}
    f(\phi)=\sum_{n=0}^\infty b_n\cos(n\phi),\quad b_n\in\mathbb{R}.
\end{equation}
The first and  second derivative of $f(\phi)$ are
\begin{equation}
    \dot{f}(\phi)=-\sum_{n=1}^\infty nb_n\sin(n\phi),\quad \ddot{f}(\phi)=-\sum_{n=1}^\infty n^2b_n\cos(n\phi).
\end{equation}
Thus
\begin{align}
    & \dot{f}^2=-\frac{1}{2}\sum_{n=1}^\infty\sum_{k=1}^\infty nkb_nb_k(\cos(n+k)\phi-\cos(n-k)\phi), \\
    & f\ddot{f}=-\frac{1}{2}\sum_{n=0}^\infty\sum_{k=1}^\infty k^2b_nb_k(\cos(n+k)\phi+\cos(n-k)\phi).
\end{align}
Set $b_{-1}=0$. Then through direct calculations, \eqref{2.31} becomes
\begin{equation}
\begin{aligned}
    &\sum_{k=1}^\infty\Big( -b_0k^2b_k+a_4b_k+\frac{1}{2}a_5b_{k-1}+\frac{1}{2}a_6b_{k-2}+\frac{1}{2}a_5b_{k+1}+\frac{1}{2}a_6b_{k+2}+2a_3b_0b_k \\
    &+\frac{1}{2}a_1(k-1)b_{k-1} +\frac{1}{2}a_2(k-2)b_{k-2}-\frac{1}{2}a_1(k+1)b_{k+1}-\frac{1}{2}a_2(k+2)b_{k+2}\Big)\cos(k\phi) \\
    &+\underbrace{\sum_{n=1}^\infty \sum_{k=1}^\infty((\frac{1}{3}n-\frac{1}{2}k)k+\frac{1}{2}a_3)b_nb_k\cos(n+k)\phi}_A  +  \underbrace{\sum_{n=1}^\infty \sum_{k=1}^\infty(\frac{1}{2}a_3-(\frac{1}{3}n+\frac{1}{2}k)k)b_nb_k\cos(n-k)\phi}_B \\
    & +\Big(\frac{3}{8}a_1a_2+\frac{1}{2}a_5b_0+\frac{1}{2}a_6b_1-\frac{1}{2}a_2b_1\Big)\cos\phi+\Big(\frac{1}{2}a_6b_0+\frac{3}{16}a_1^2\Big)\cos2\phi+\frac{3}{8}a_1a_2\cos3\phi\\
    &+\frac{3}{16}a_2^2\cos4\phi
     +a_3b_0^2-\frac{1}{2}a_1b_1-a_2b_2+\frac{1}{2}a_6b_2+\frac{1}{2}a_5b_1+a_4b_0-\frac{3}{16}(a_1^2+a_2^2)=0. \label{3.5}
\end{aligned}
\end{equation}
\\ \\
One also computes $A$ and $B$ as
\begin{equation*}
\begin{aligned}
    A & =\sum_{m=2}^\infty \cos(m\phi)(\sum_{k=1}^{m-1}(\frac{1}{3}(m-k)-\frac{1}{2}k)k+\frac{1}{2}a_3)b_kb_m\\
    &=\sum_{n=2}^\infty \cos(n\phi)(\sum_{k=1}^{n-1}(\frac{1}{2}a_3+k(\frac{1}{3}n-\frac{5}{6}k))b_kb_{n-k})\\
    & =\sum_{l=2}^\infty \cos(l\phi)(\sum_{k=1}^{l-1}(\frac{1}{2}a_3+k(\frac{1}{3}l-\frac{5}{6}k))b_kb_{l-k}),\\
    B & =\sum_{k=1}^\infty(\frac{1}{2}a_3-\frac{5}{6}k^2)b_k^2+\sum_{l=1}^\infty\sum_{k=1}^l(\frac{1}{2}a_3-(\frac{1}{3}(k+l)+\frac{1}{2}k)k)b_{k+l}b_k)\cos(l\phi)\\
    & \ \ \ +\sum_{l=1}^\infty\sum_{k=l+1}^\infty((\frac{1}{2}a_3-(\frac{1}{3}(k+l)+\frac{1}{2}k)k)b_{k+l}b_k+(\frac{1}{2}a_3-(\frac{1}{3}(k-l)+\frac{1}{2}k)k)b_{k-l}b_k)\cos(l\phi)\\
    &=\sum_{n=1}^\infty(\frac{1}{2}a_3-\frac{5}{6}n^2)b_n^2\\
    &\quad+\sum_{l=1}^\infty\left(\sum_{k=1}^\infty(\frac{1}{2}a_3-(\frac{1}{3}l+\frac{5}{6}k)k)b_{k+l}b_k+\sum_{k=l+1}^\infty(\frac{1}{2}a_3-(\frac{5}{6}k-\frac{1}{3}l)k)b_{k-l}b_k\right)\cos(l\phi).
\end{aligned}
\end{equation*}
Reorganizing equation \eqref{3.5}, one has the following:

\begin{framed}
\begin{equation}
    A_0+\sum_{l=1}^\infty A_l\cos(l\phi)+\sum_{l=2}^\infty B_l\cos(l\phi)+C_1\cos\phi+C_2\cos(2\phi)+C_3\cos(3\phi)+C_4\cos(4\phi)=0,\label{3.6}
\end{equation}
\end{framed}
in which
\begin{equation*}
    \begin{aligned}
    & A_0=\sum_{n=1}^\infty(\frac{1}{2}a_3-\frac{5}{6}n^2)b_n^2+a_3b_0^2-\frac{1}{2}a_1b_1-a_2b_2+\frac{1}{2}a_6b_2+\frac{1}{2}a_5b_1+a_4b_0-\frac{3}{16}(a_1^2+a_2^2),\\
    & A_l=\sum_{k=1}^\infty(\frac{1}{2}a_3-(\frac{1}{3}l+\frac{5}{6}k)k)b_{k+l}b_k+\sum_{k=l+1}^\infty(\frac{1}{2}a_3-(\frac{5}{6}k-\frac{1}{3}l)k)b_{k-l}b_k \\
    &~~~~~~~~\qquad - b_0l^2b_l+a_4b_l+\frac{1}{2}a_5b_{l-1}+\frac{1}{2}a_6b_{l-2}+\frac{1}{2}a_5b_{l+1}+\frac{1}{2}a_6b_{l+2}+2a_3b_0b_l \\
    & ~~~~~~~~\qquad+\frac{1}{2}a_1(l-1)b_{l-1} +\frac{1}{2}a_2(l-2)b_{l-2}-\frac{1}{2}a_1(l+1)b_{l+1}-\frac{1}{2}a_2(l+2)b_{l+2},\\
    & B_l=\sum_{k=1}^{l-1}(\frac{1}{2}a_3+(\frac{1}{3}l-\frac{5}{6}k)k)b_kb_{l-k},\\
    & C_1=-\frac{3}{8}a_1a_2+\frac{1}{2}a_5b_0+\frac{1}{2}a_6b_1-\frac{1}{2}a_2b_1,\\
    & C_2=\frac{1}{2}a_6b_0+\frac{3}{16}a_1^2,\quad C_3=\frac{3}{8}a_1a_2,\quad C_4=\frac{3}{16}a_2^2.
\end{aligned}
\end{equation*}

Now we do truncation for $f(\phi)$ with term $N$, i.e., we set $b_{N+1}=b_{N+2}=\cdots=0$. Then
\eqref{3.6} yields that
\begin{equation}
    \left\{
    \begin{aligned}
        & A_0=0,   \\
        & A_1+C_1=0, \\
        & A_i+B_i+C_i=0,\quad i=2,3,4, \\
        & A_j+B_j=0,\quad   j=5,\ldots,N.
    \end{aligned}
    \right.
\end{equation}
We denote this system as
\begin{equation}
    F(b)=(F_0(b),F_1(b),\ldots,F_N(b))=0,\quad b=(b_0,b_1,\ldots,b_N). \label{3.8}
\end{equation}

We then need to solve this non-linear system numerically.

\subsection{Newton's method}
Newton's method is a classical numerical algorithm to solve non-linear equations, say
\begin{equation*}
    f(x)=0.
\end{equation*}
It is based on Taylor's formula of the function
\begin{equation*}
    f(x)=f(x_0)+f'(x_0)(x-x_0)+\frac{1}{2}f''(x_0)(x-x_0)^2+R_n(x-x_0)^2,
\end{equation*}
and its idea is regarding the solution to the linear equation
\begin{equation*}
    f(x_0)+f'(x_0)(x_1-x_0)=0,\quad \text{namely}\quad  x_1=x_0-\frac{f(x_0)}{f'(x_0)}
\end{equation*}
as an approximate solution to the non-linear equation. Then we repeat this procedure from point $x_1$ and get $x_2$, from $x_{n-1}$ and get $x_n$, until $|x_n-x_{n-1}|$ is tolerably small. For the system
\begin{equation*}
    F(x)=0,
\end{equation*}
we only need to change the derivative $f'(x)$ into the Jacobian matrix $\mathrm{D}F$,
\begin{equation}
    \mathrm{D}F=\Big(\frac{\partial F_i}{\partial b_j}\Big)_{i,j=0,\dots,N},
\end{equation}
and the rest is the same. See Algorithm \ref{al1}.

\begin{algorithm}[]  
	\caption{Newton's method for system \eqref{3.8}.}\label{al1}
	\LinesNumbered 
	\KwIn{$b_0$, $y_0$, $max$, $tol$;}
	\KwOut{$b_0$, $y_0$;}
	$b_0$ is the guessing initial point of $F$,~$y_0=F(b_0)$. $max$ is iteration times and $tol$ is tolerance\; 
	\For{i=1: max}{
		$J\gets \mathrm{D}F(b_0)$\;
        $b_1\gets b_0-J^{-1}y_0$\;
        \If{$|b_1-b_0|\geq tol$}{
            $y_0\gets F(b_1)$\;
            $b_0\gets b_1$\;
        }
    }
\end{algorithm}

\subsection{Numerical results for $N=5,6,7,8,9,10$} In the Figures \ref{fig2} and \ref{fig3},  we show numerical results for $f, f'$ and $W_C$, with the semi-vertex angle of the cone being $\theta_0=\frac{\pi}{6}$, and various attack angles $\alpha_0$ up to $\tfrac{\pi}{6}$ for $N=5$ and $10$. In each figure, the graph on the left is the numerical solution of $f(\phi)$ and its derivative $f'(\phi)$, and on the right is the pressure $W_C(\phi)$.

\begin{figure}[h]
    \centering
    \subfigure[$\alpha_0=\frac{\pi}{36}$]{
    \begin{minipage}[b]{.3\linewidth}
    \centering
    \includegraphics[scale=0.13]{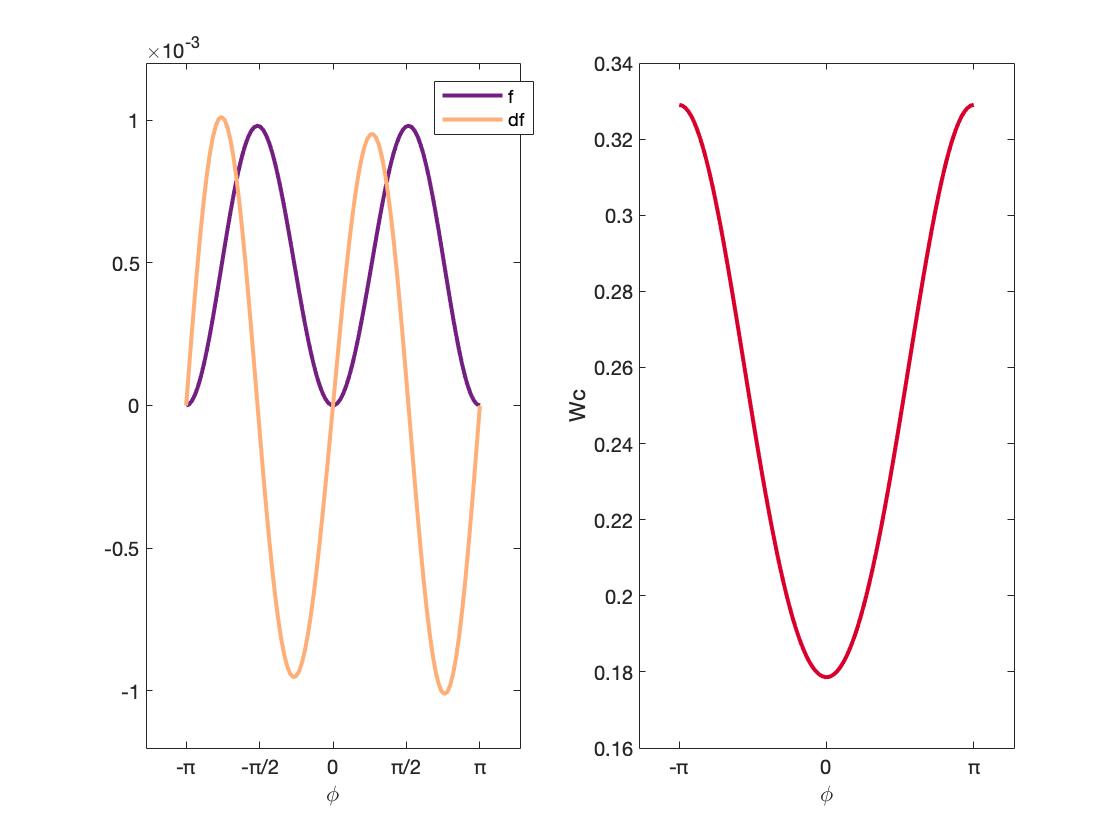}
    \end{minipage}
    }
    \subfigure[$\alpha_0=\frac{\pi}{24}$]{
    \begin{minipage}[b]{.3\linewidth}
    \centering
    \includegraphics[scale=0.13]{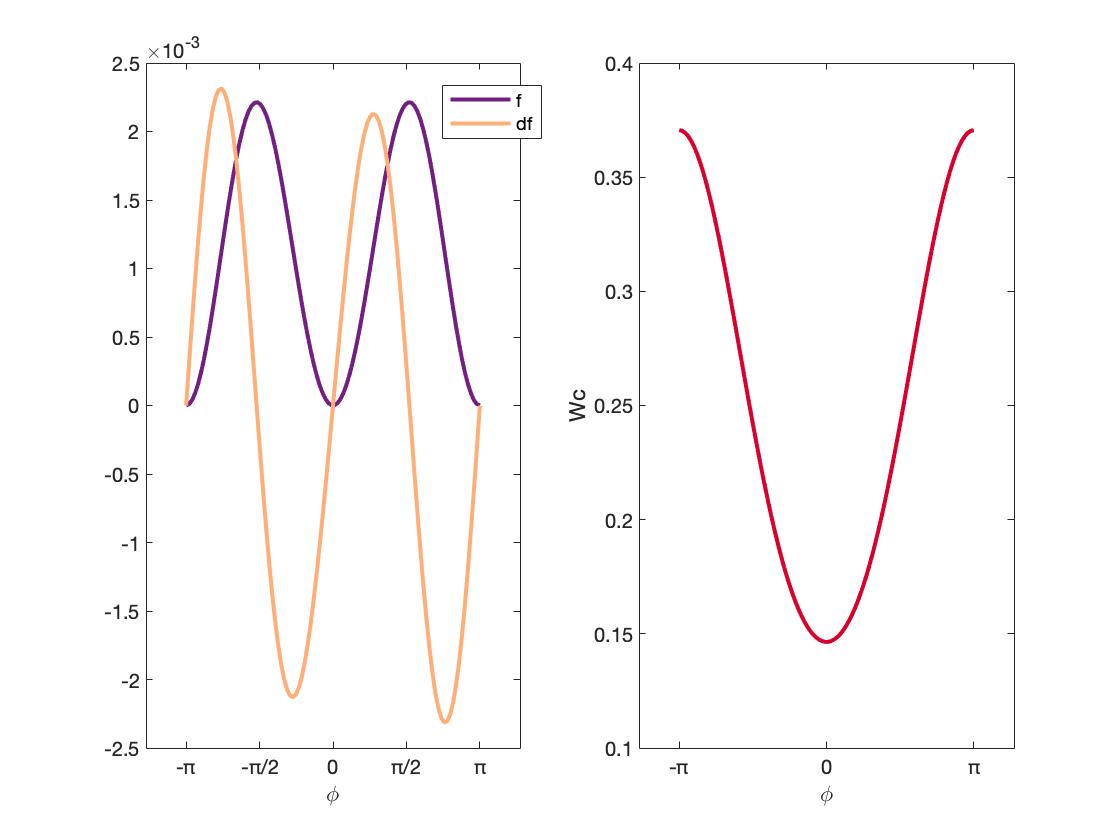}
    \end{minipage}
    }
    \subfigure[$\alpha_0=\frac{\pi}{18}$]{
    \begin{minipage}[b]{.3\linewidth}
    \centering
    \includegraphics[scale=0.13]{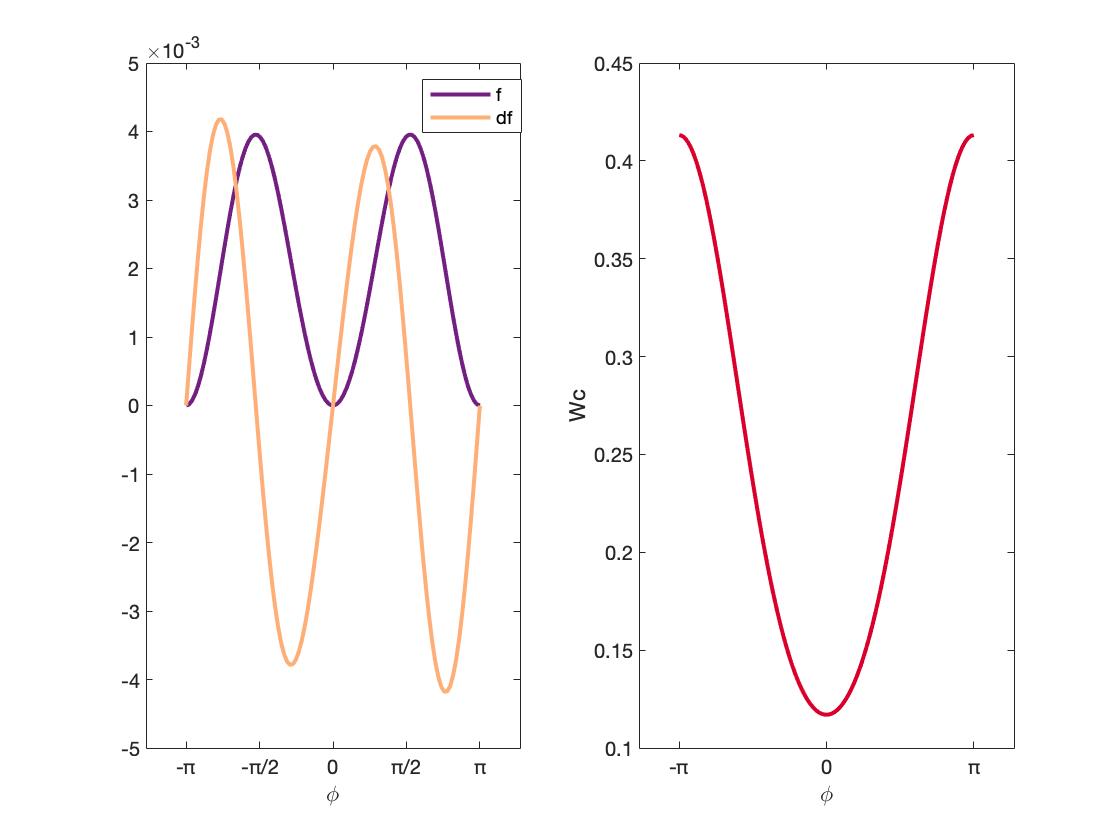}
    \end{minipage}
    }
    \subfigure[$\alpha_0=\frac{\pi}{12}$]{
    \begin{minipage}[b]{.3\linewidth}
    \centering
    \includegraphics[scale=0.13]{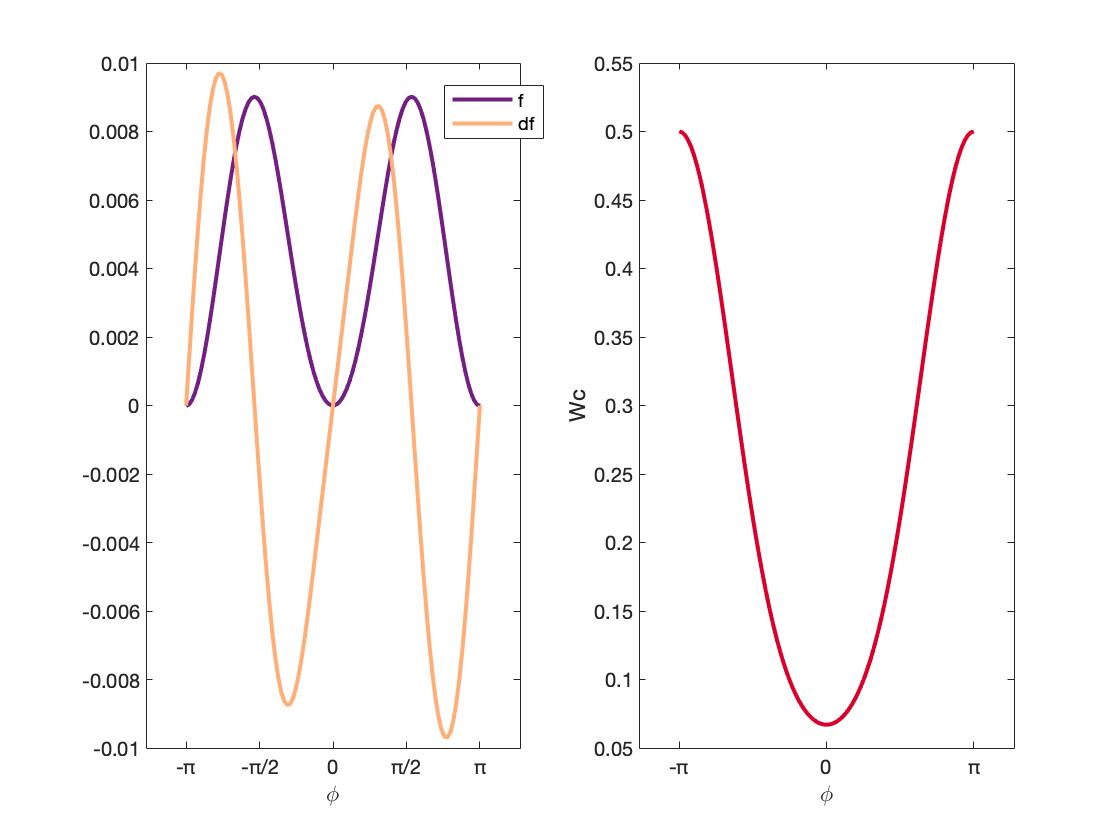}
    \end{minipage}
    }
    \subfigure[$\alpha_0=\frac{\pi}{9}$]{
    \begin{minipage}[b]{.3\linewidth}
    \centering
    \includegraphics[scale=0.13]{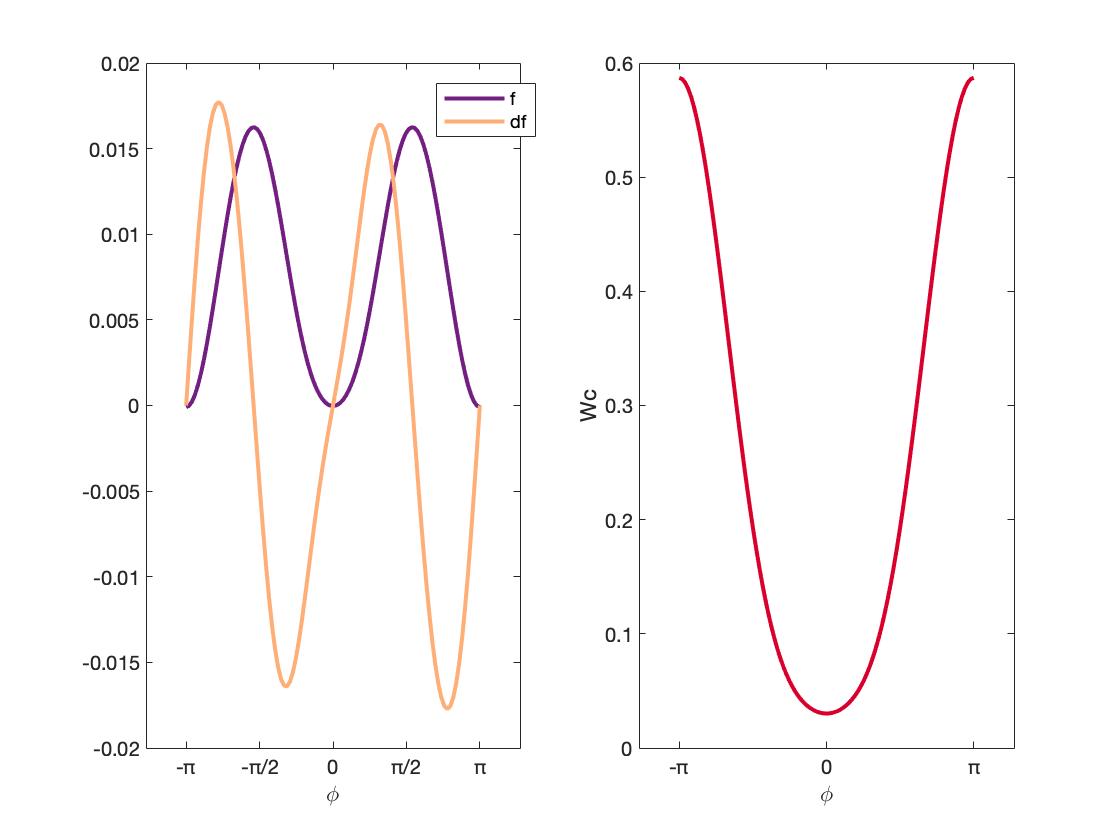}
    \end{minipage}
    }
    \subfigure[$\alpha_0=\frac{\pi}{6}$]{
    \begin{minipage}[b]{.3\linewidth}
    \centering
    \includegraphics[scale=0.13]{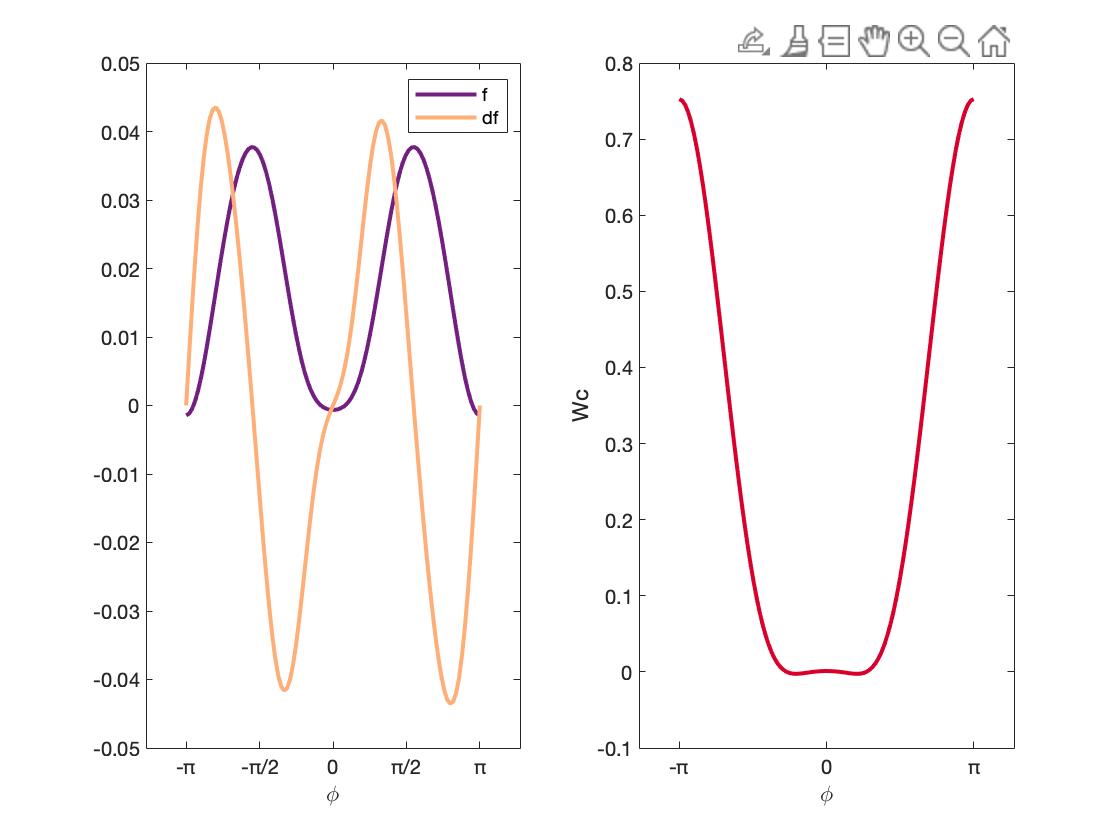}
    \end{minipage}
    }
    \caption{$f, \dot{f}$ and $W_C$ for  $N=5,~\theta_0=\frac{\pi}{6}$.}\label{fig2}
\end{figure}

\begin{figure}[h]
    \centering
    \subfigure[$\alpha_0=\frac{\pi}{36}$]{
    \begin{minipage}[b]{.3\linewidth}
    \centering
    \includegraphics[scale=0.13]{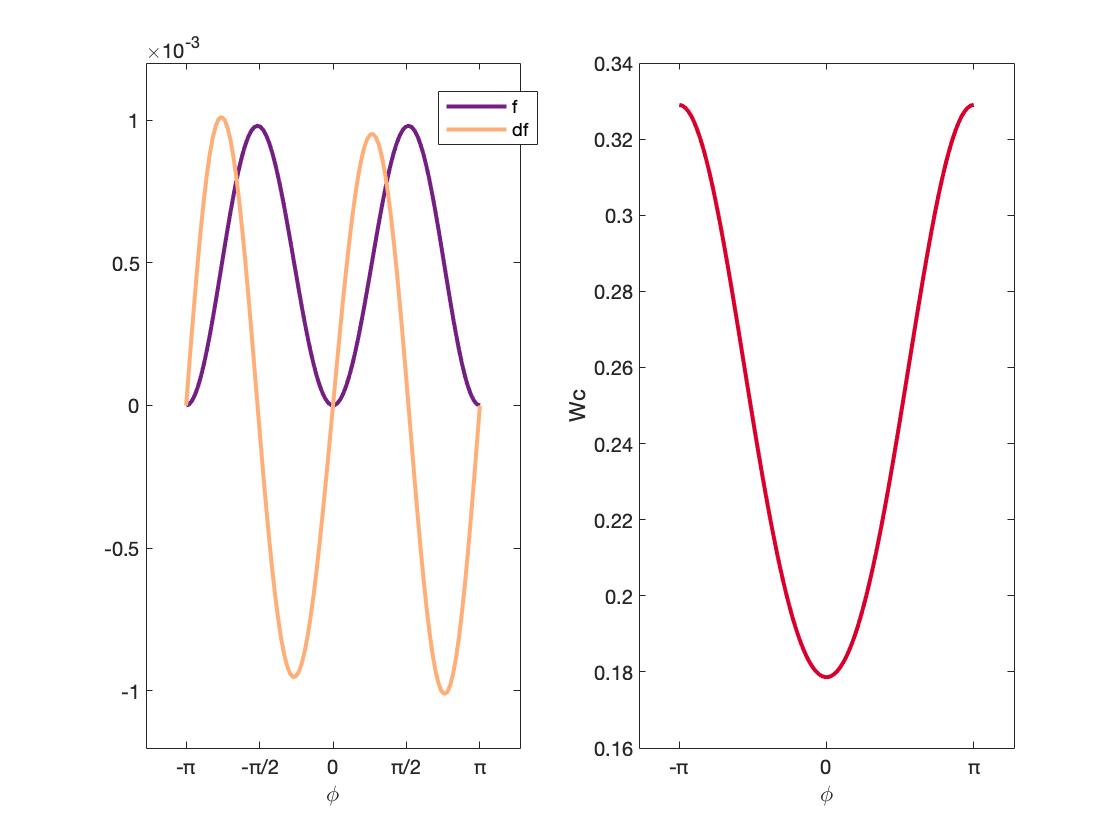}
    \end{minipage}
    }
    \subfigure[$\alpha_0=\frac{\pi}{24}$]{
    \begin{minipage}[b]{.3\linewidth}
    \centering
    \includegraphics[scale=0.13]{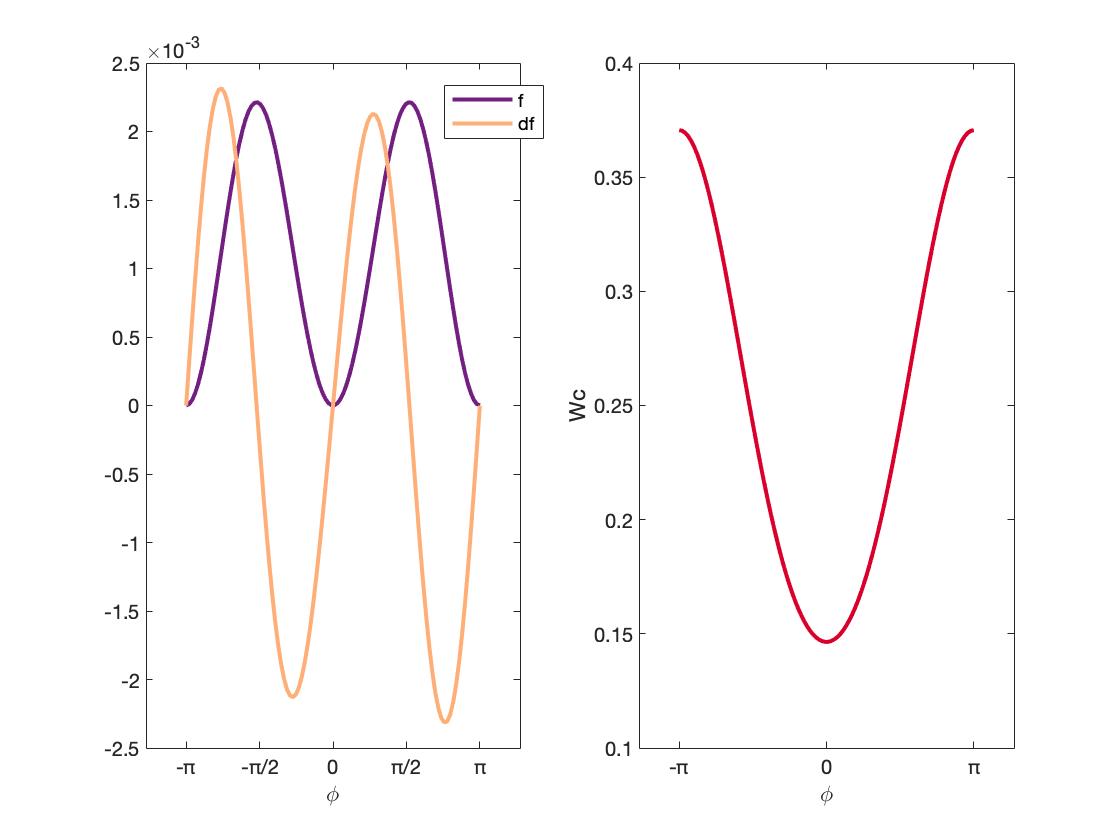}
    \end{minipage}
    }
    \subfigure[$\alpha_0=\frac{\pi}{18}$]{
    \begin{minipage}[b]{.3\linewidth}
    \centering
    \includegraphics[scale=0.13]{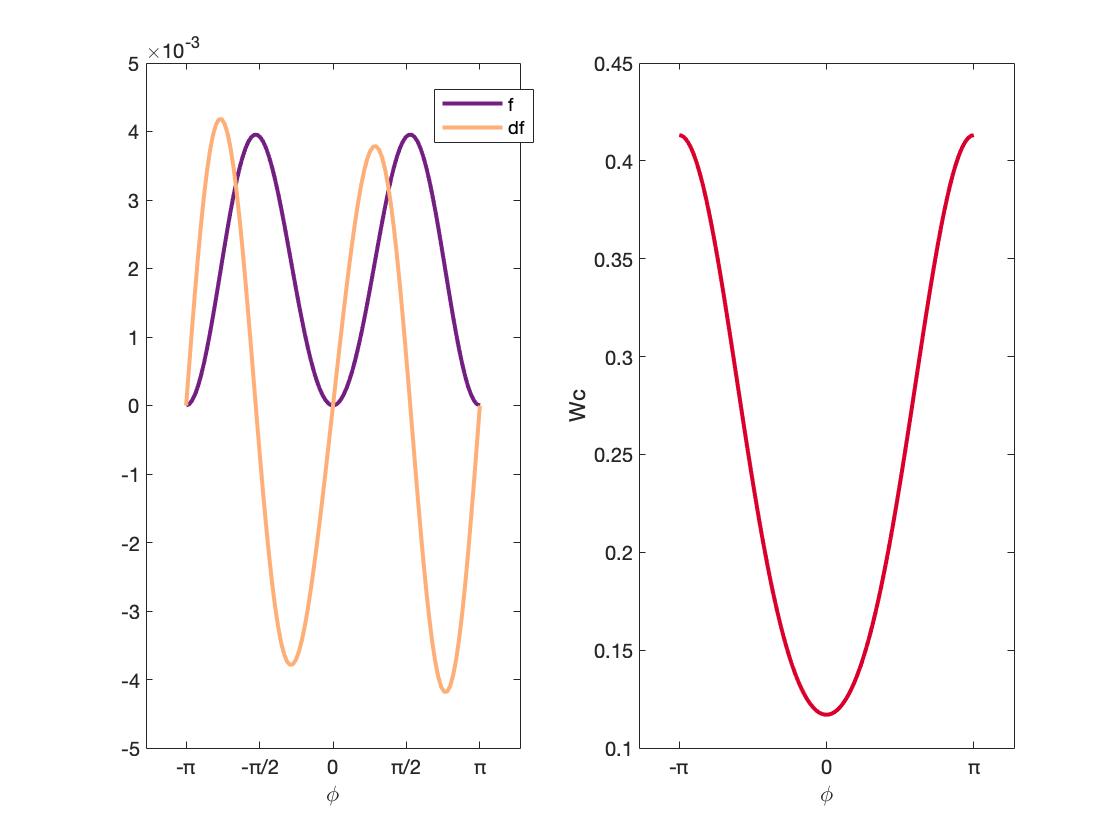}
    \end{minipage}
    }
    \subfigure[$\alpha_0=\frac{\pi}{12}$]{
    \begin{minipage}[b]{.3\linewidth}
    \centering
    \includegraphics[scale=0.13]{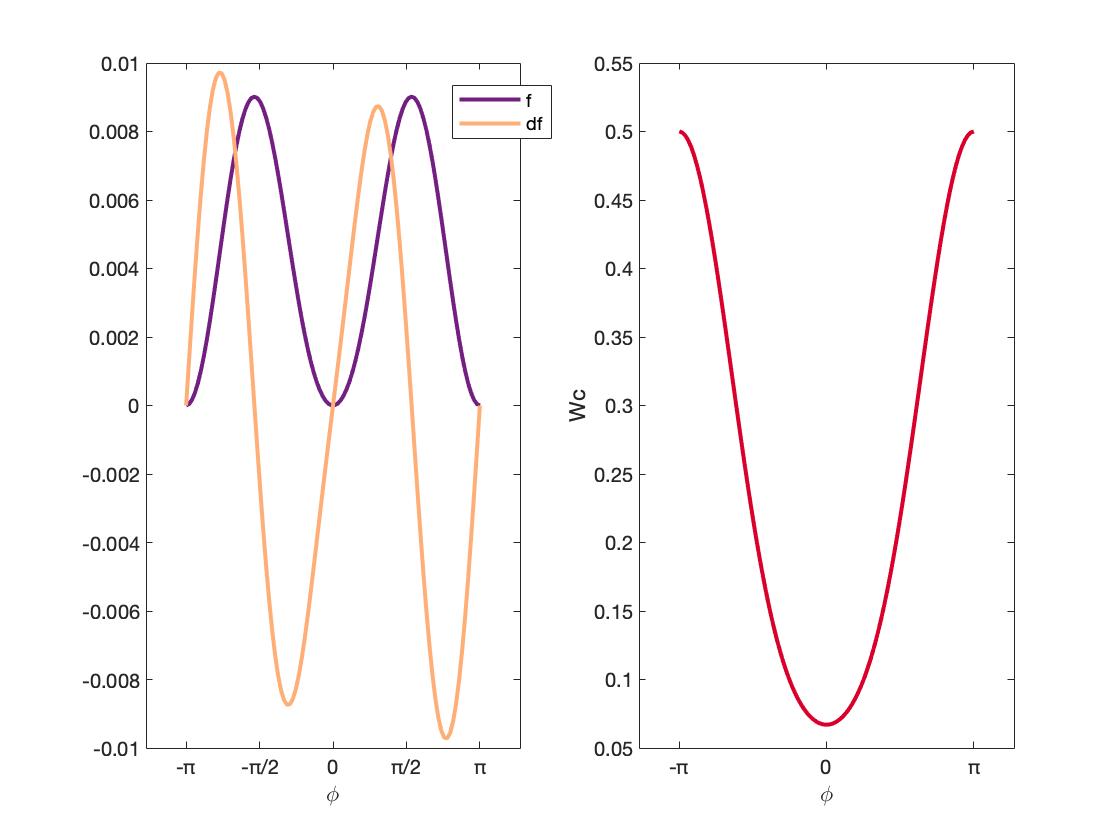}
    \end{minipage}
    }
    \subfigure[$\alpha_0=\frac{\pi}{9}$]{
    \begin{minipage}[b]{.3\linewidth}
    \centering
    \includegraphics[scale=0.13]{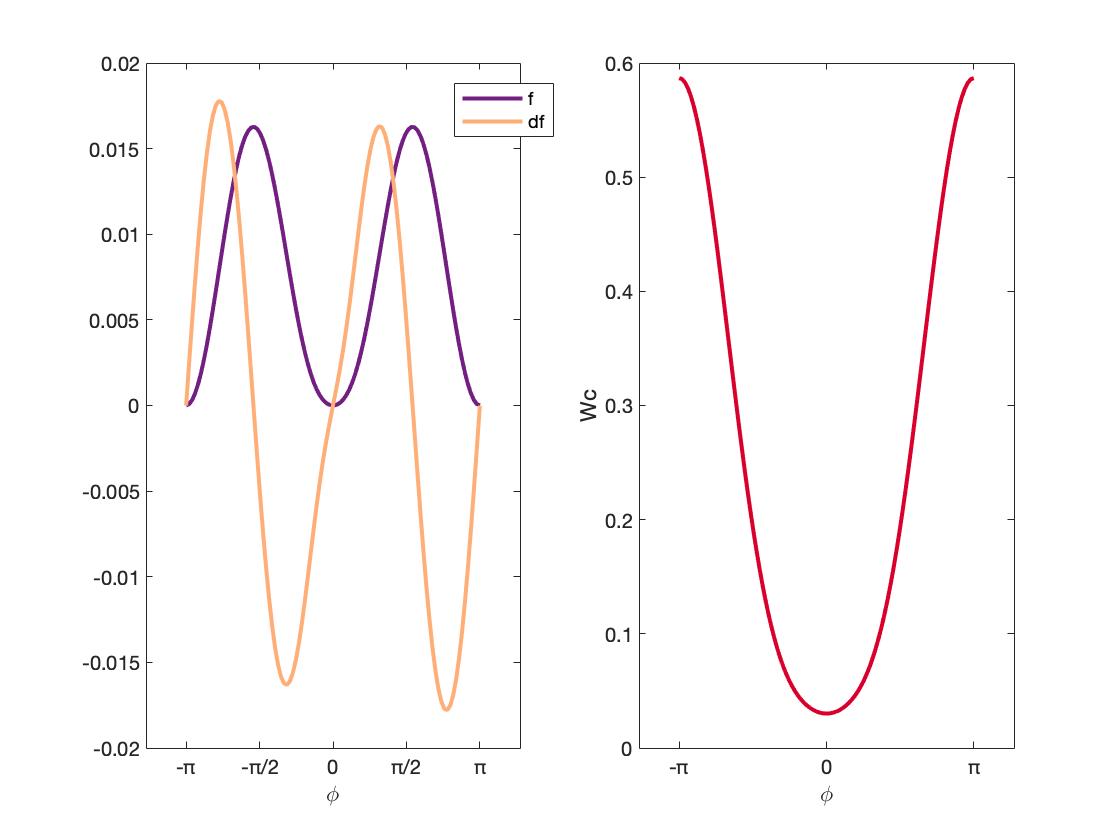}
    \end{minipage}
    }
    \subfigure[$\alpha_0=\frac{\pi}{6}$]{
    \begin{minipage}[b]{.3\linewidth}
    \centering
    \includegraphics[scale=0.13]{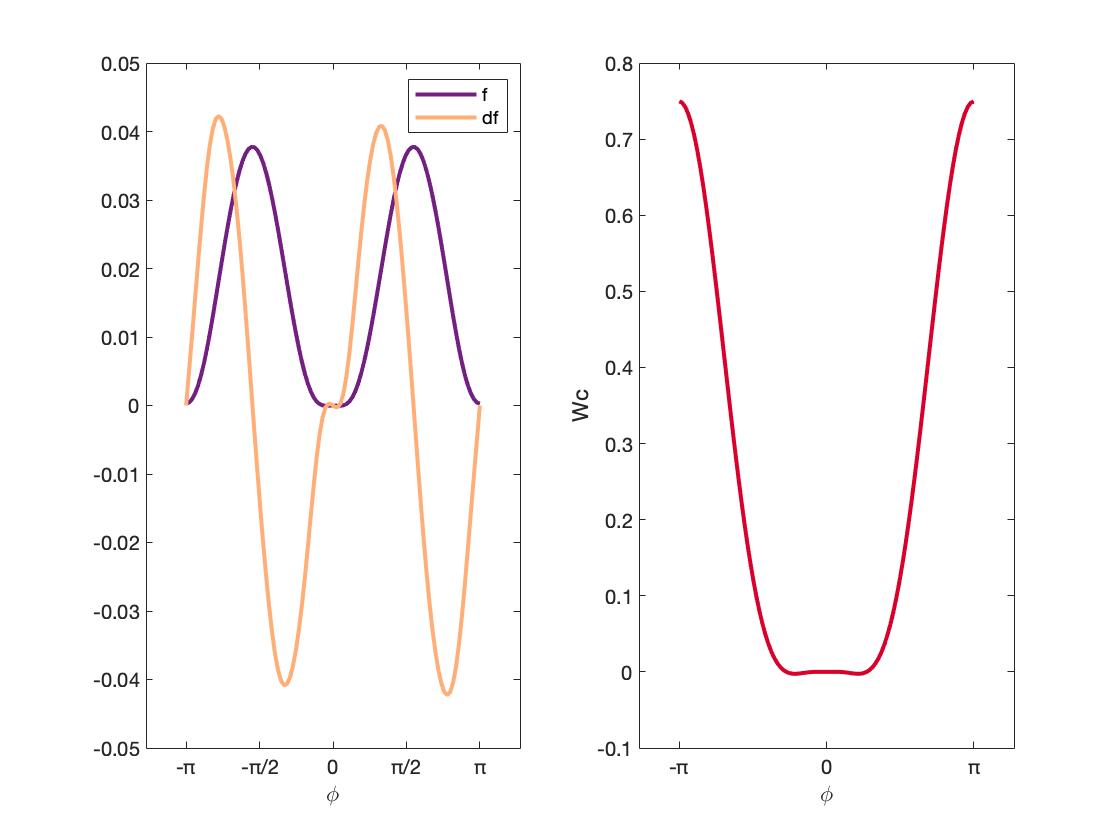}
    \end{minipage}
    }
    \caption{$f,\dot{f}$ and $W_C$ for $N=10,~\theta_0=\frac{\pi}{6}$.}\label{fig3}
\end{figure}

 We notice that as shown in \eqref{2.32}, $W_C(\phi)$ is influenced by the upstream data $\Tilde{p}(\rho_0, E_0)$. In these numerical experiments, for simplicity, we had assumed that $\Tilde{p}(\rho_0, E_0)=0$ (taking hypersonic limit for the polytropic gases or Chaplygin gas, i.e.,  $M_\infty=\infty$).
 
 In Figures \ref{fig2} and \ref{fig3}, it is seen that when $\alpha_0$ takes small values (say $\alpha_0\leq\frac{\pi}{9}$), $f$ is indeed a periodic function with $f\geq0,~f(-\pi)=f(0)=f(\pi)=0$, and $\dot{f}(-\pi)=\dot{f}(0)=\dot{f}(\pi)=0$. We also see that $W_C(\phi)>0$, which is a necessary condition for infinite-thin shock layer solution. However, when $\alpha_0$ is large, say $\alpha_0\geq\frac{\pi}{6}$, mistakes occur on $f,~\dot{f} $ and $W_C$ (see pictures (f) in Figures \ref{fig2} and \ref{fig3}), indicating that solutions with mass concentration on the surface is not suitable for large attack angles.

We observe from these two pictures that for the same attack angle $\alpha_0$, the curves for $N=5$ and $10$ are virtually identical, implying that only tiny error exists between diverse values of $N$. To demonstrate this, we put the curves for $N$ from $5$ to $10$ in one   figure and get Figure \ref{fig4} . We see that it looks only one curve for each $f,\dot{f}$ and $W_C$ appearing in the figure,  showing that the results for different $N$ nearly coincide with each other. To further analyze the error, see Section 4.4. 
\begin{figure}[htbp]
    \begin{minipage}[t]{0.5\linewidth}
        \centering
        \includegraphics[scale=0.19]{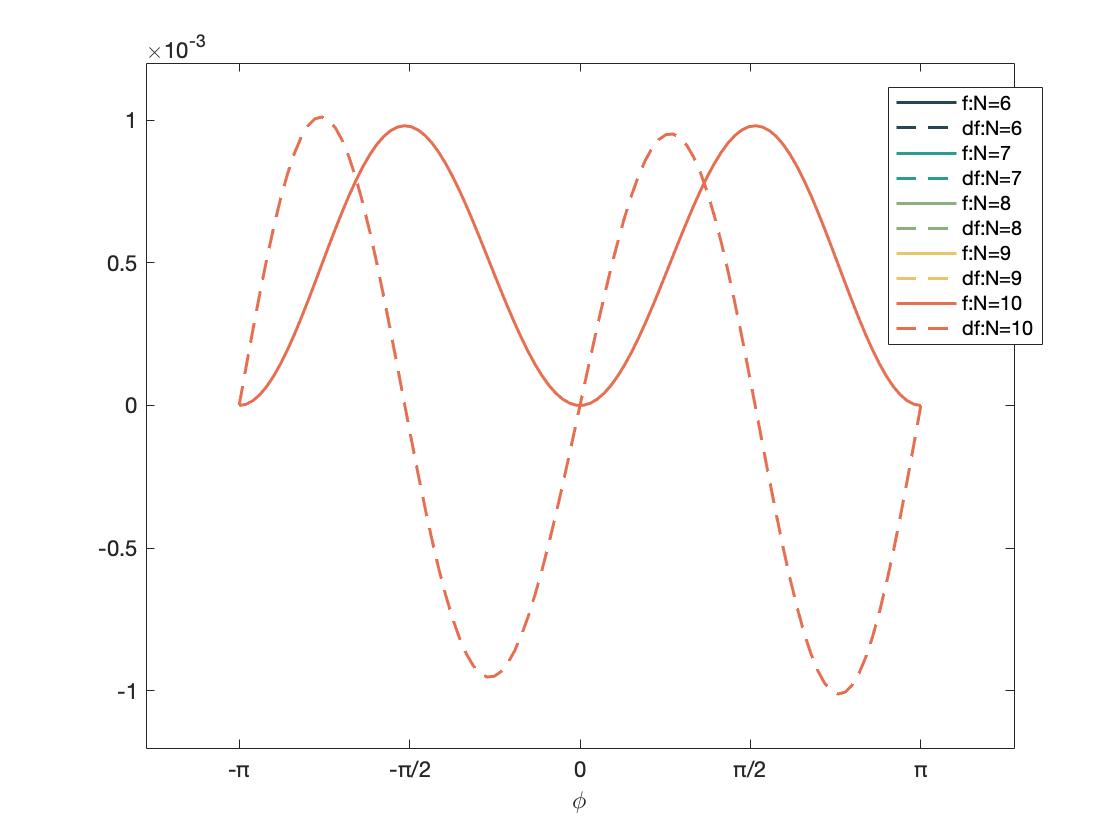}
        \centerline{(a) $f$ and $\dot{f}$ for different $N$.}
    \end{minipage}%
    \begin{minipage}[t]{0.5\linewidth}
        \centering
        \includegraphics[scale=0.19]{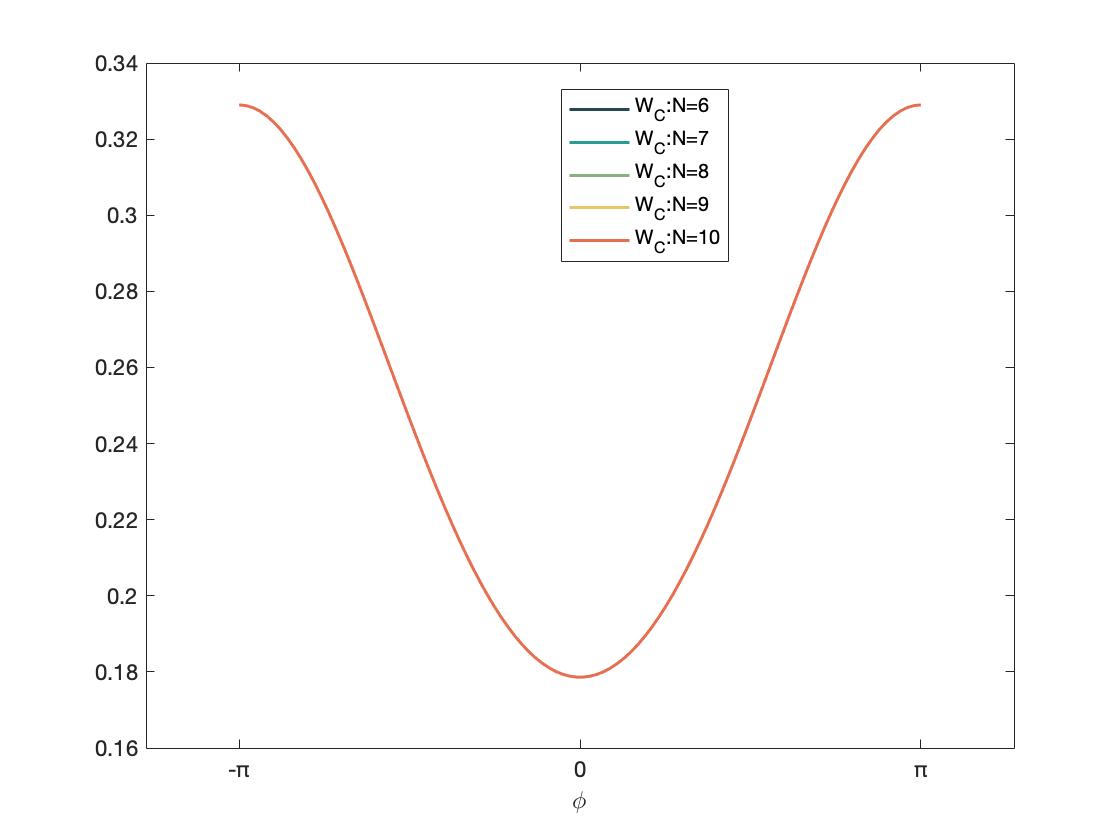}
        \centerline{(b) $W_C$ for different $N$.}
    \end{minipage}
    \caption{$f,\dot{f}$ and $W_C$ for different $N$ when $\theta_0=\frac{\pi}{6},~\alpha_0=\frac{\pi}{36}$.}\label{fig4}
\end{figure}

\begin{remark}
    In Figures \ref{fig2} and \ref{fig3}, we carefully observe that the peak value of $\dot{f}(\phi)$ on $[-\pi,-\frac{\pi}{2}]$ is a bit larger than that on $[0,~\frac{\pi}{2}]$. This is because the particles on the lower semi-cone (which is the windward side) are more affected by the upcoming flow than the upper semi-cone (leeward side), so the kinetic energy changes faster.
\end{remark}

\subsection{Error Estimates} Figure  \ref{fig5}   shows the error of system \eqref{3.6} with different truncation terms $N$.

\begin{figure}[htbp]
\begin{minipage}[t]{0.33\linewidth}
        \centering
        \includegraphics[scale=0.13]{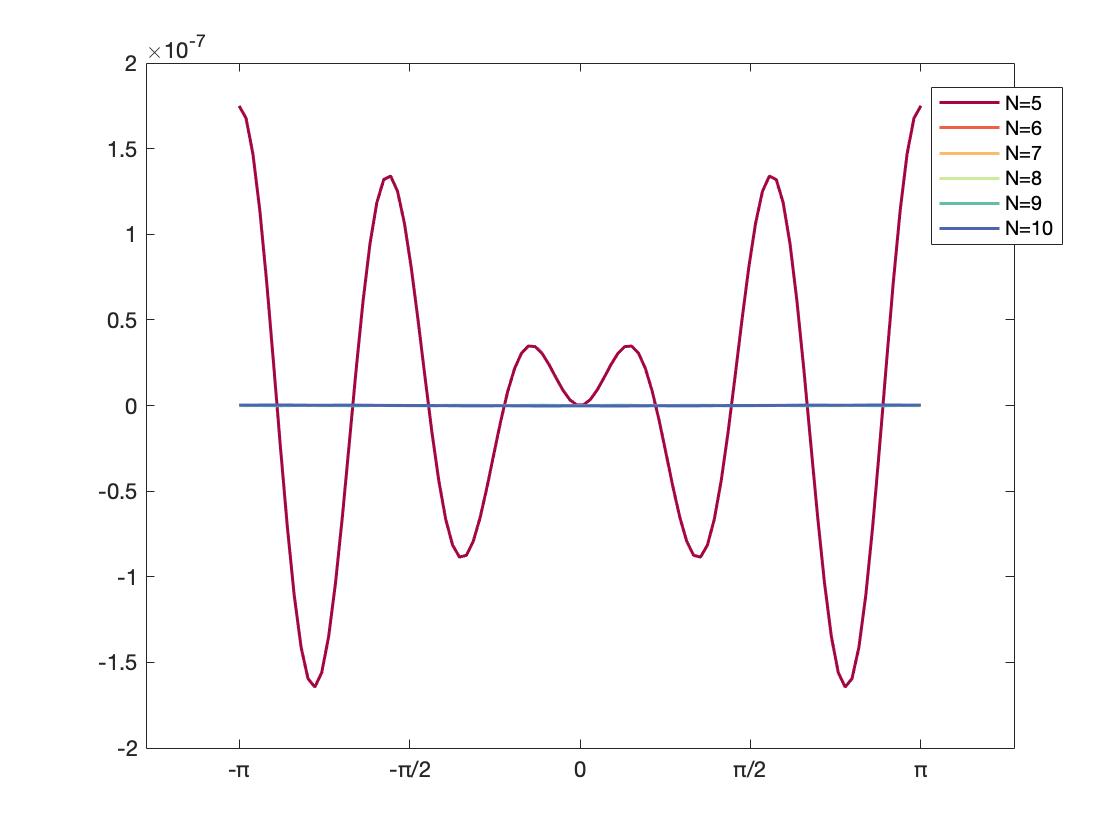}
        \centerline{(a) $E_N$ for $N=5$.}
    \end{minipage}%
    \begin{minipage}[t]{0.33\linewidth}
        \centering
        \includegraphics[scale=0.13]{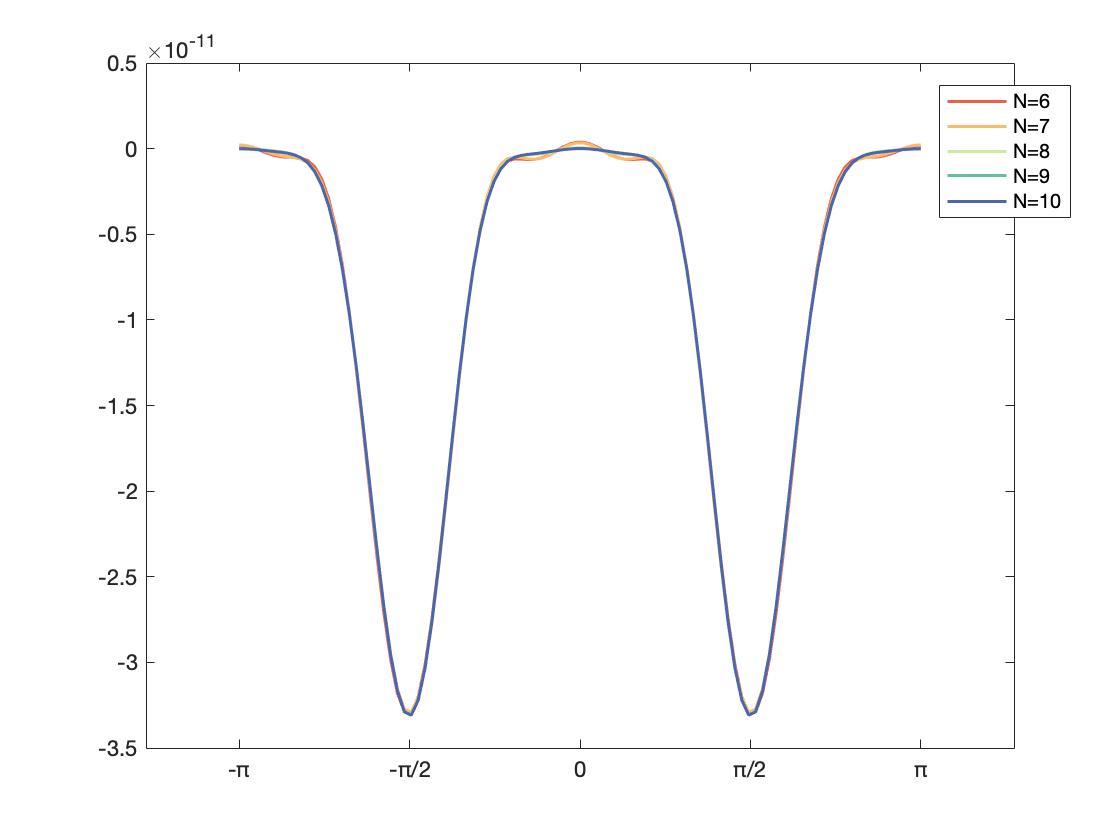}
        \centerline{(b) $E_N$ for  $N=6, 7, 8, 9, 10$.}
    \end{minipage}
     \begin{minipage}[t]{0.33\linewidth}
\centering
\includegraphics[scale=0.13]{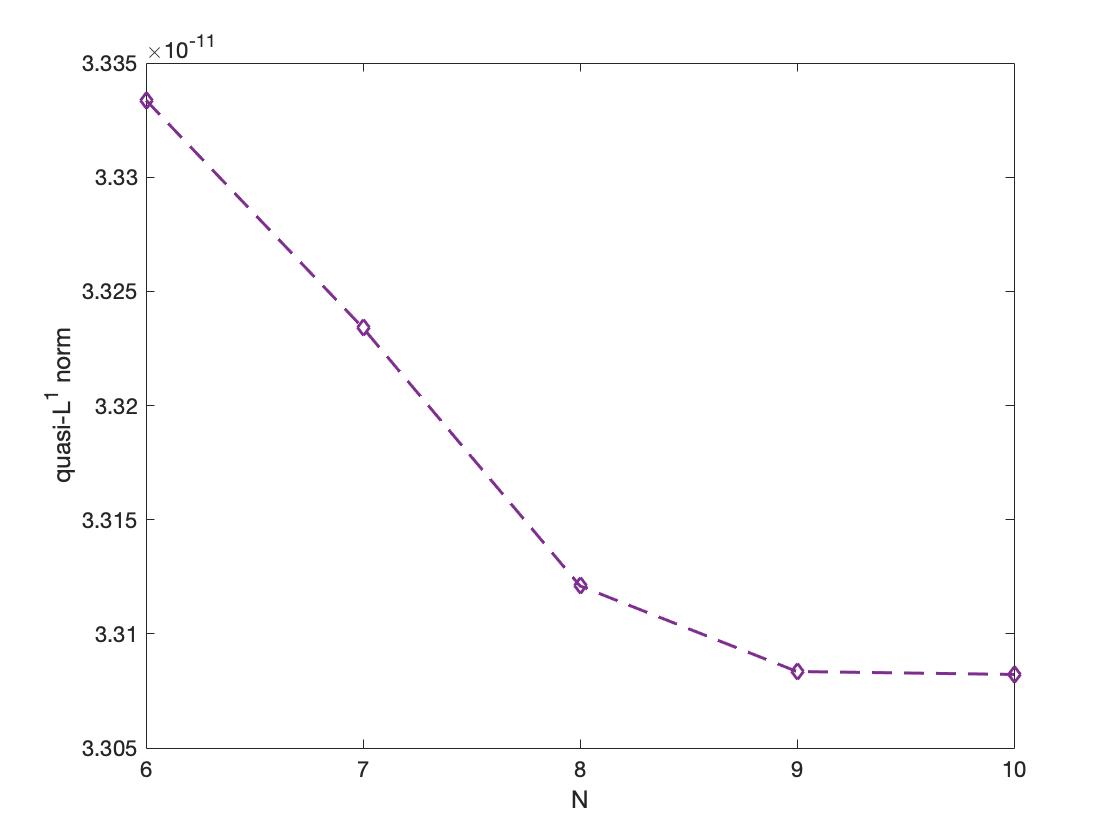}
\centerline{(c) Quasi-$L^1$ Norm of each $N$.}\label{fig7}
\end{minipage}
\caption{Error Estimates}\label{fig5}
\end{figure}
The error $E_N(\phi)$ is obtained by substituting $A_0,\ldots, A_N,~B_2,\dots,B_N,C_1,C_2,C_3,C_4$ into the left-hand side of \eqref{3.6} for each $N=5, 6, \ldots$ Figure \ref{fig5} (a) (b) shows the corresponding results for $N\geq5$, from which we observe that when $N=5$, the error is under $2\times10^{-7}$, while when $N\geq6$, the error is comparatively much smaller, fluctuates four orders smaller ($4\times10^{-11}$). For $N\ge6$ and up to $10$,  the curves almost coincide with each other, implying that there is not much difference between each such $N$.
 
    In order to see whether the coupling of Fourier spectral method and Newton's method is a good treatment to system \eqref{3.6}, we introduce a concept of \textbf{quasi-$L^1$ norm} of the error, defined as the sum of  $\max\{E_N(\phi)^+:~\phi\in[-\pi,\pi]\}$ and  $\max\{E_N(\phi)^-:~\phi\in[-\pi,\pi]\}$ with different $N$, where $E_N^+$ is the positive part of $E_N$ and $E_N^-$ is the negative part of $E_N$. The result is shown in Figure \ref{fig5} (c). We see that as $N$ takes larger value, the error becomes smaller, indicating a good astringency of our numerical method.

    In view of the fact that the order of error $(10^{-11})$ is sufficiently small compared to the order of $f~(10^{-4})$ , we will not distinguish the value of $N$ in the following series of numerical results.
 
\subsection{Results for various semi-vertex angles} Figure \ref{fig6} shows the results for different semi-vertex angles $\theta_0$ for one attack angle $\alpha_0$ when $\alpha_0=\pi/36$ and $\pi/18$. An obvious property of monotonicity is observed from the figure. For the same attack angle $\alpha_0$, $f(\phi)$ takes larger value when $\theta_0$ increases, and so is $W_C(\phi)$.

\begin{figure}[htbp]
    \begin{minipage}[t]{0.5\linewidth}
        \centering
        \includegraphics[scale=0.19]{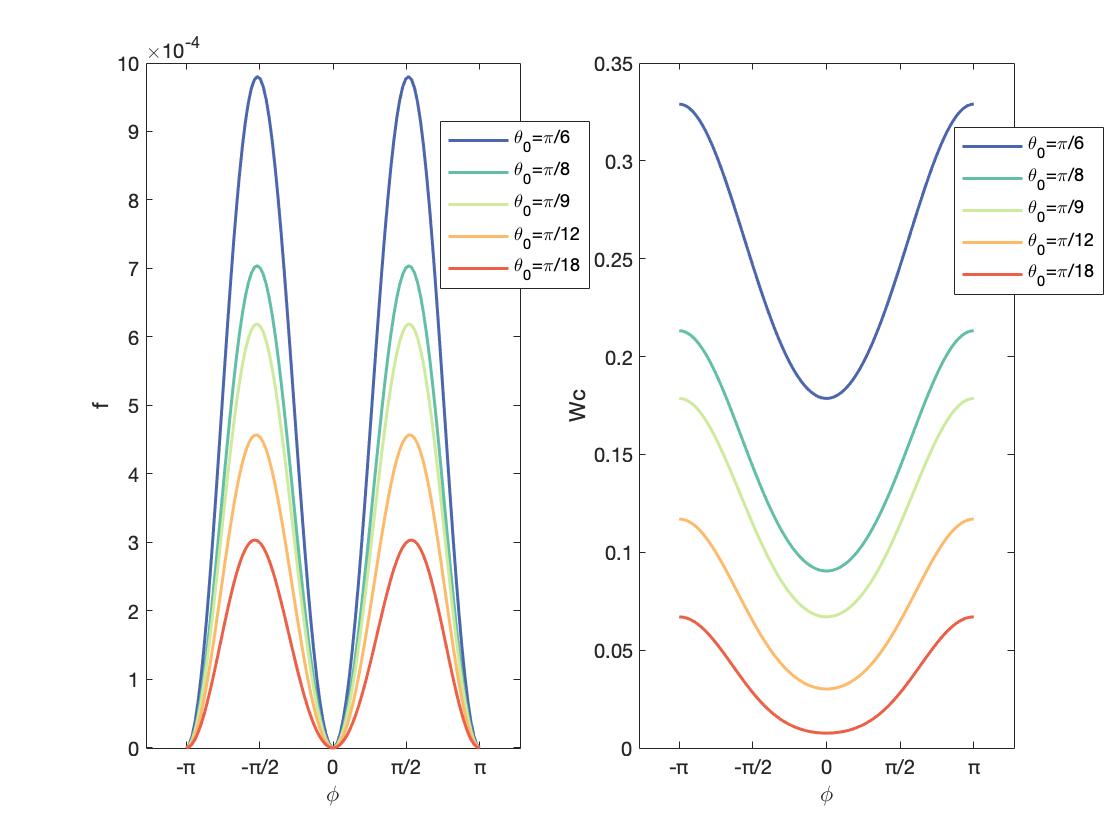}
        \centerline{(a) $\alpha_0=\pi/36$.}
    \end{minipage}%
    \begin{minipage}[t]{0.5\linewidth}
        \centering
        \includegraphics[scale=0.19]{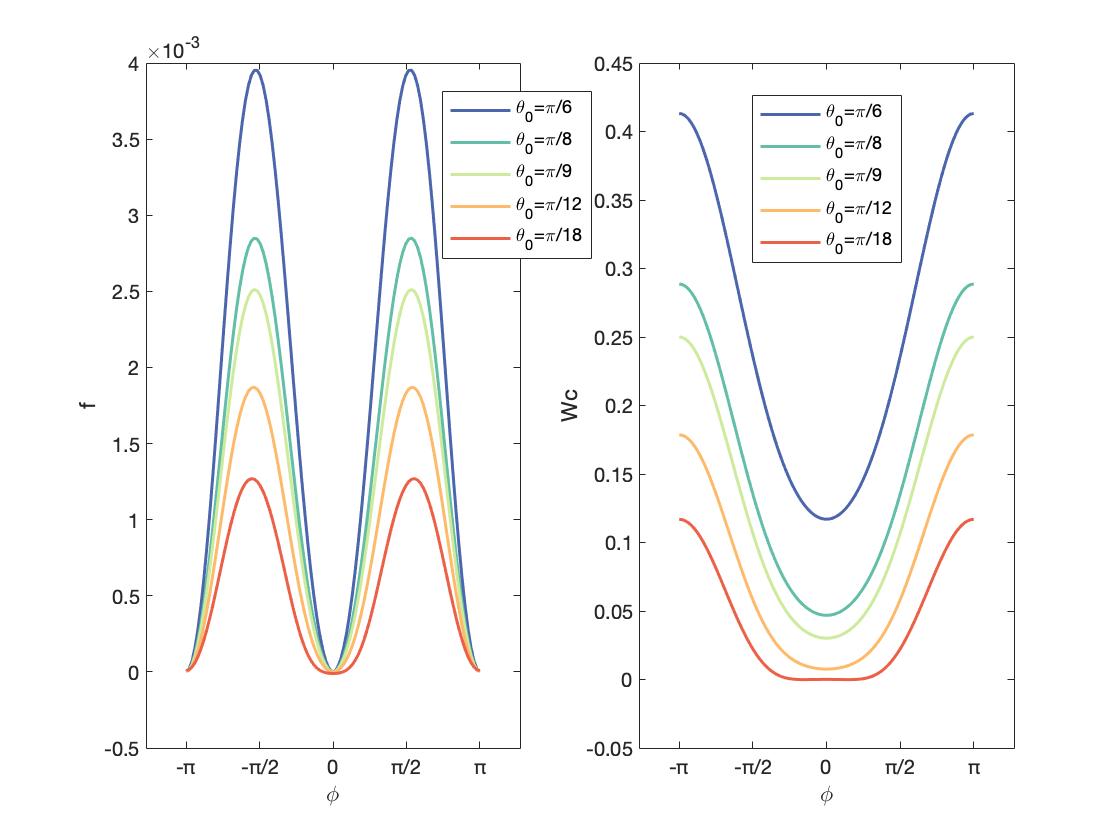}
        \centerline{(b) $\alpha_0=\pi/18$.}
    \end{minipage}
    \caption{Different semi-vertex angles $\theta_0$ for the same attack angle.}\label{fig6}
\end{figure}

 \subsection{Results for various attack angles and extremum for $W_C$.} Figure \ref{fig7} is the results of $f$ and $W_C$ for fixed $\theta_0=\frac{\pi}{6}$, and various $\alpha_0$. We also acquire   monotonicity for $f$. It is seen that $f(\phi)$ increases when $\alpha_0$ takes larger value. For $W_C$, however, $W_C(\phi)$ is monotonically increasing with respect to $\alpha_0$ when $\phi$ is near $\pm\pi$,  while it is  monotonically decreasing when $\phi$ is near $0$. To reveal the variation rule of $W_C$, we extract the extremum of $W_C$ with each $\alpha_0$ and get Figure \ref{fig8}.  We observed that as $\alpha_0$ increases, $\max W_C$ increases while $\min W_C$ decreases. It is compatible with the physical fact that as $\alpha_0$ increases (in a reasonable range), windward side ($\phi=\pm\pi$) bears more pressure while leeward side $(\phi=0)$ bears less.

\begin{figure}[htbp]
\begin{minipage}[t]{0.5\linewidth}
\centering
\includegraphics[scale=0.2]{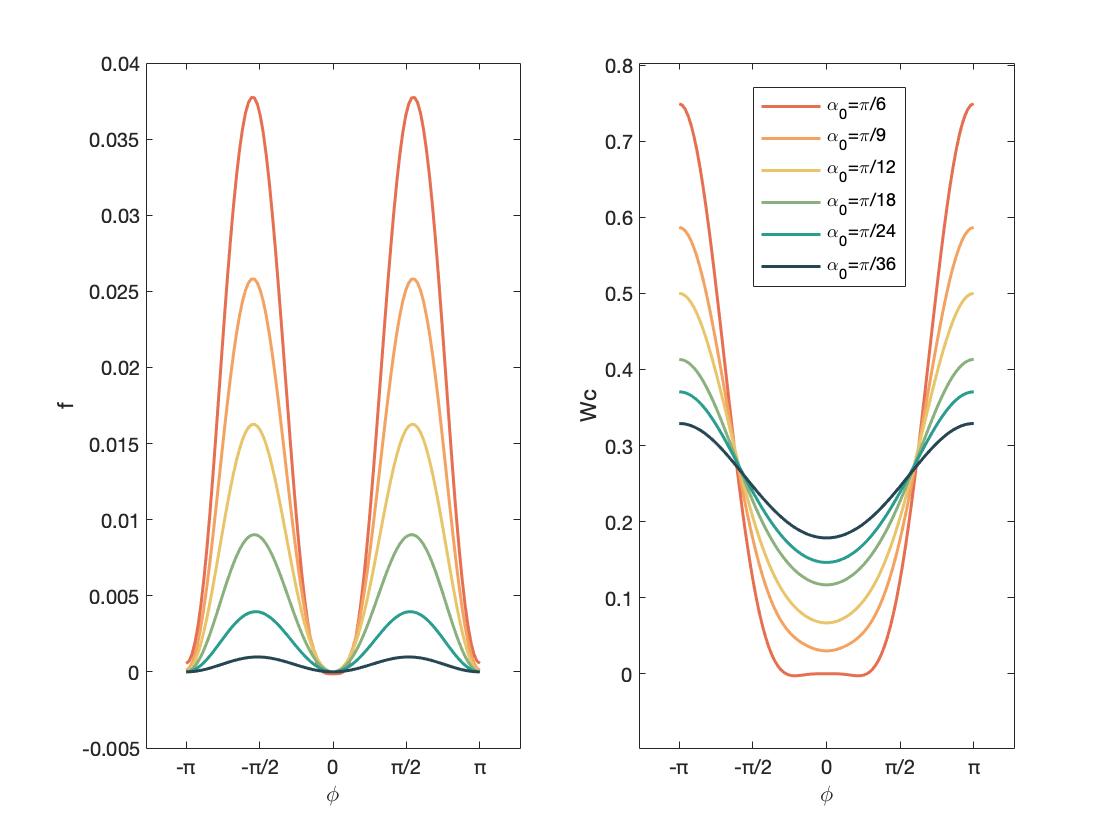}
\caption{Different attack angles $\alpha_0$ for $\theta_0=\frac{\pi}{6}.$}\label{fig7}
\end{minipage}%
\begin{minipage}[t]{0.5\linewidth}
\centering
\includegraphics[scale=0.2]{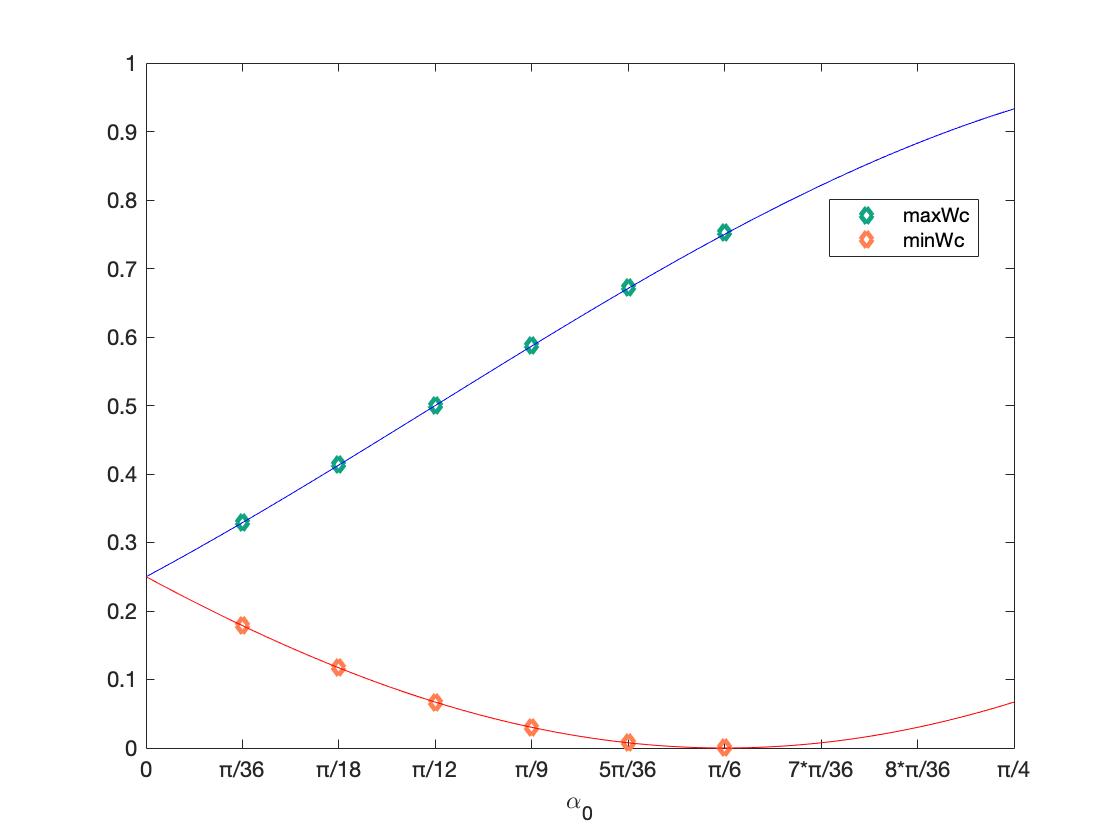}
\caption{Maximum and Minimum of $W_C$.}\label{fig8}
\end{minipage}
\end{figure}

\begin{remark}
    The blue curve in Figure \ref{fig8} is the graph of $\sin^2(\frac{\pi}{6}+\alpha_0)$ and the red curve is $~\sin^2(\frac{\pi}{6}-\alpha_0)$, where $\alpha_0$ is the variable. We see that $\max W_C$ and $\min W_C$ exactly fall on these two curves, and this is consistent with the theoretical result \eqref{2.32}.
\end{remark}
\subsection{Results for $u^t,~w$ and $w_\rho$} Using Fourier spectral method, 
$f$ and $\dot{f}$ is obtained numerically, and  $h$ is also known by \eqref{2.29}. Then by solving \eqref{2.28} with, for example,  Runge-Kutta method, we get the numerical results for $y$. Therefore,  according to \eqref{1}, $u^t,~w$ and $w_\rho$ can also be solved numerically ($u^t=f/y, ~w=h/y,~w_\rho=y^2/f$). In Figure \ref{fig9}, we show some results for different attack angle $\alpha_0$ with semi-vertex angle $\theta_0=\frac{\pi}{6}$. In (b) and (c), singularity arises at 
$\phi=0,~\pm\pi$ because $y=0$ at these three points. We attribute this error to the numerical method because $w$ and $w_\rho$   should be $C^1$ functions over $[-\pi, ~\pi]$.

\begin{figure}[htbp]
    \begin{minipage}[t]{0.33\linewidth}
        \centering
        \includegraphics[scale=0.15]{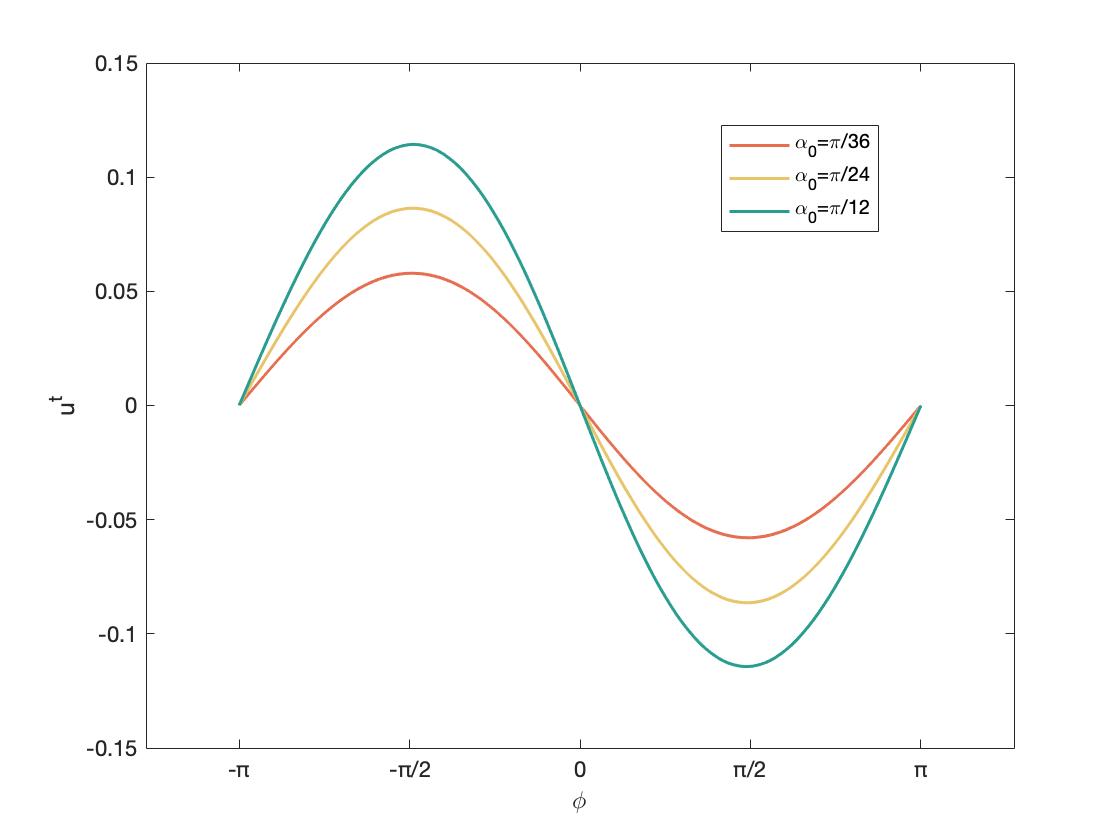}
        \centerline{(a) Some results for $u^t$.}
    \end{minipage}%
    \begin{minipage}[t]{0.33\linewidth}
        \centering
        \includegraphics[scale=0.15]{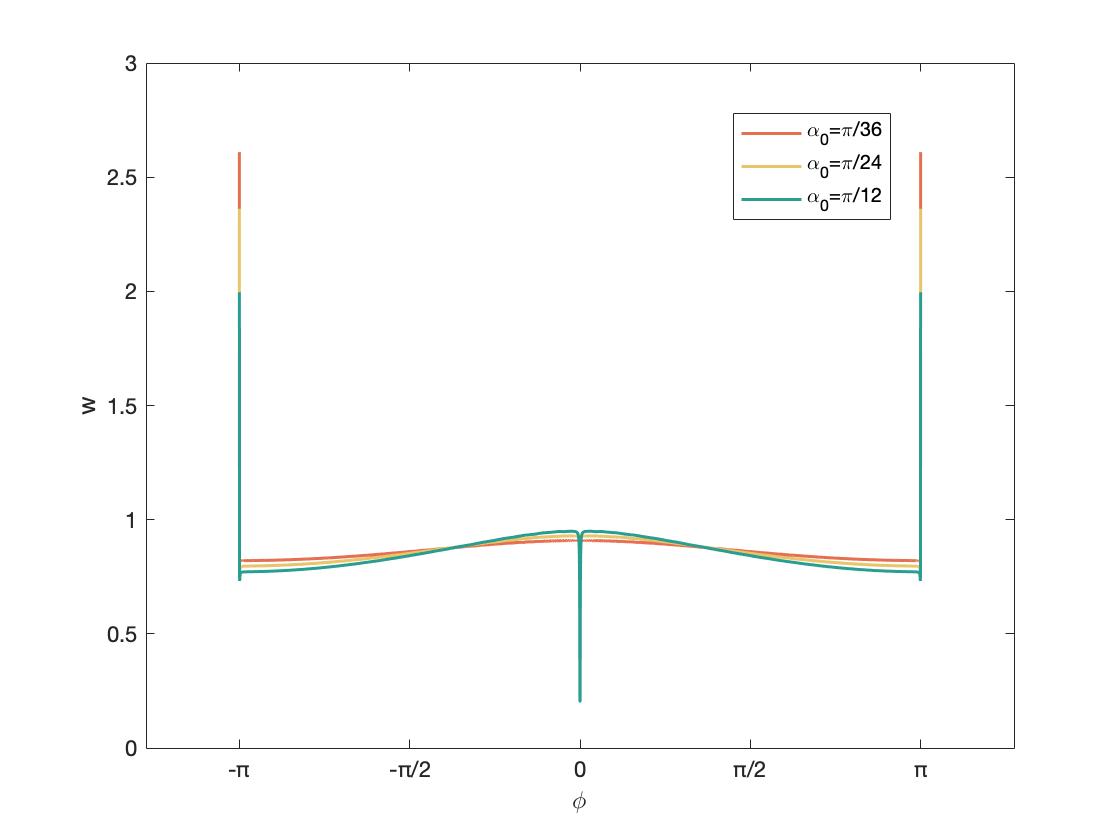}
        \centerline{(b) Some results for $w$.}
    \end{minipage}
    \begin{minipage}[t]{0.33\linewidth}
        \centering
        \includegraphics[scale=0.15]{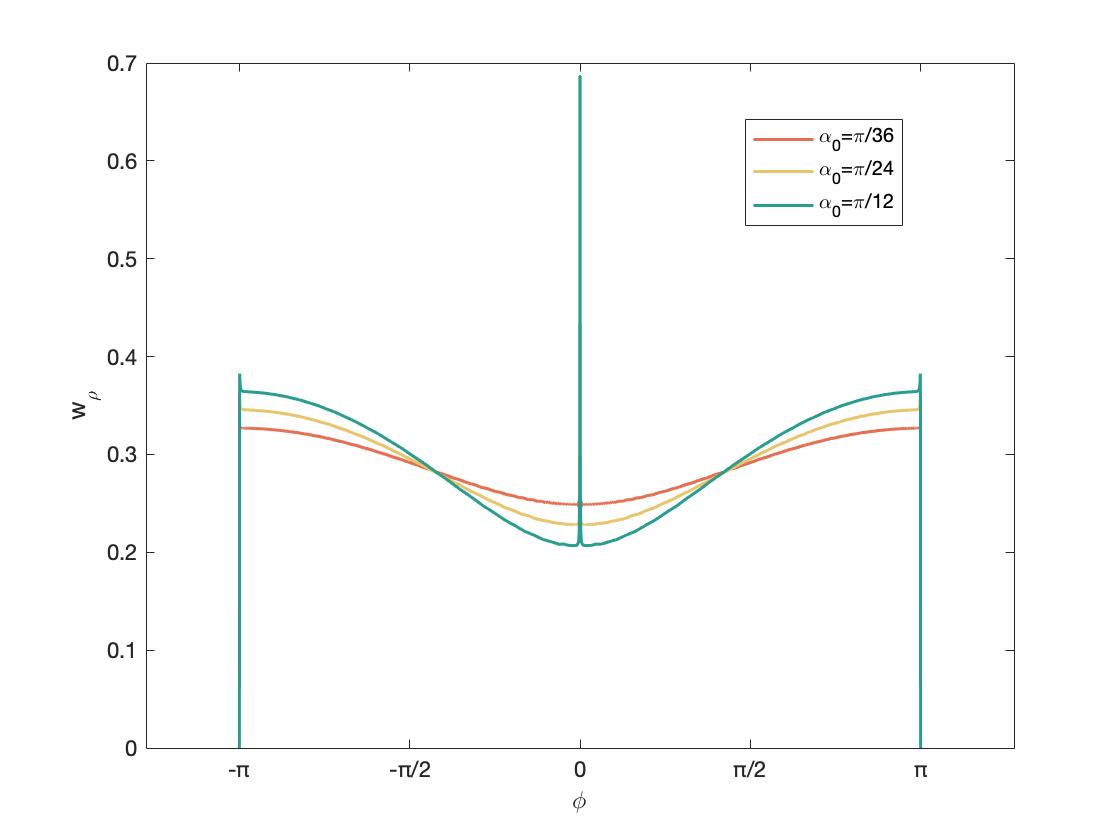}
        \centerline{(c) Some results for $w_\rho$.}
    \end{minipage}
    \caption{$u^t,~w$ and $w_\rho$ for $\theta_0=\frac{\pi}{6}$ and different $\alpha_0$.}\label{fig9}
\end{figure}
\begin{remark}
    Figure \ref{fig9} (a) indicates that $u^t$ is not equal to the component of the upcoming stream along $\mathbf{t}$, which, by \cite{1}, is $(u_0,\mathbf{t})=\sin\alpha_0\sin\phi$. This is because the particles are affected not only by the cone, but also by the constant upcoming flow hitting on the cone.
\end{remark}
\begin{remark}
     Figure \ref{fig9} (c) shows that $w_\rho$ is monotone with respect to $\alpha_0$ near $\phi=\pm\pi$ and $0$ respectively. When $\phi$ is near $\pm\pi$, $w_\rho(\phi)$ is monotonically increasing, since more particles (per unit square) will hit the cone as the attack angle increases. When $\phi$ is near $0$, $w_\rho(\phi)$ is monotonically decreasing, because the   angle between the upstream and the generatrix $\phi=0$ gets smaller as $\alpha_0$ increases, so there will be less particles hitting on this position.
\end{remark}
\subsection{Trajectories on the cone}
Suppose the trajectory of particles on the cone is given by $z=(\theta_0, \phi, r(\phi))$. For time variable $t$, there holds
\begin{equation*}
    \frac{\mathrm{d}z}{\mathrm{d}t}=\frac{\mathrm{d}z}{\mathrm{d}\phi}\frac{\mathrm{d}\phi}{\mathrm{d}t}=(\frac{\mathrm{d}\theta_0}{\mathrm{d}t},\frac{\mathrm{d}\phi}{\mathrm{d}t}, \frac{\mathrm{d}r(\phi)}{\mathrm{d}t})=(0,~u^t,~w),
\end{equation*}
therefore
\begin{equation*}
    \frac{\mathrm{d}z}{\mathrm{d}\phi}=(0,~1,~\frac{w}{u^t}).
\end{equation*}
So for some initial point $z_0=(\theta_0,\phi_0,r(\phi_0)),$ its trajectory on the cone is given by 
\begin{equation}
    z(\phi)=(\theta_0,~\phi,~ r(\phi_0)+\int_{\phi_0}^\phi \frac{w}{u^t}\mathrm{d}\phi).
\end{equation}
In Figure \ref{fig10}, we present three trajectories ($z^1$ in blue, $z^2$ in red, $z^3$ in yellow) with different initial positions from the west semi-cone (half cone containing $(0,0,-1)$), which is  $z_0^1=(\theta_0, -0.999\pi, 10),~z_0^2=(\theta_0, -\frac{3\pi}{4}, 10),~z_0^3=(\theta_0,-\frac{\pi}{2}, 10).$ Figure 10 (b) (c) (d) are views along $x^3$-axis, $x^2$-axis and $x^1$-axis respectively. We learn from the figures that the trajectories of particles with all initial positions, except for $\phi_0=-\pi$,  will gradually rotate on the cone and flow towards $\phi=0$. Due to symmetry along $x^1$-axis, so is the east semi-cone.
\begin{figure}[h]
    \centering
    \subfigure[ ]{
    \begin{minipage}[b]{.45\linewidth}
    \centering
    \includegraphics[scale=0.19]{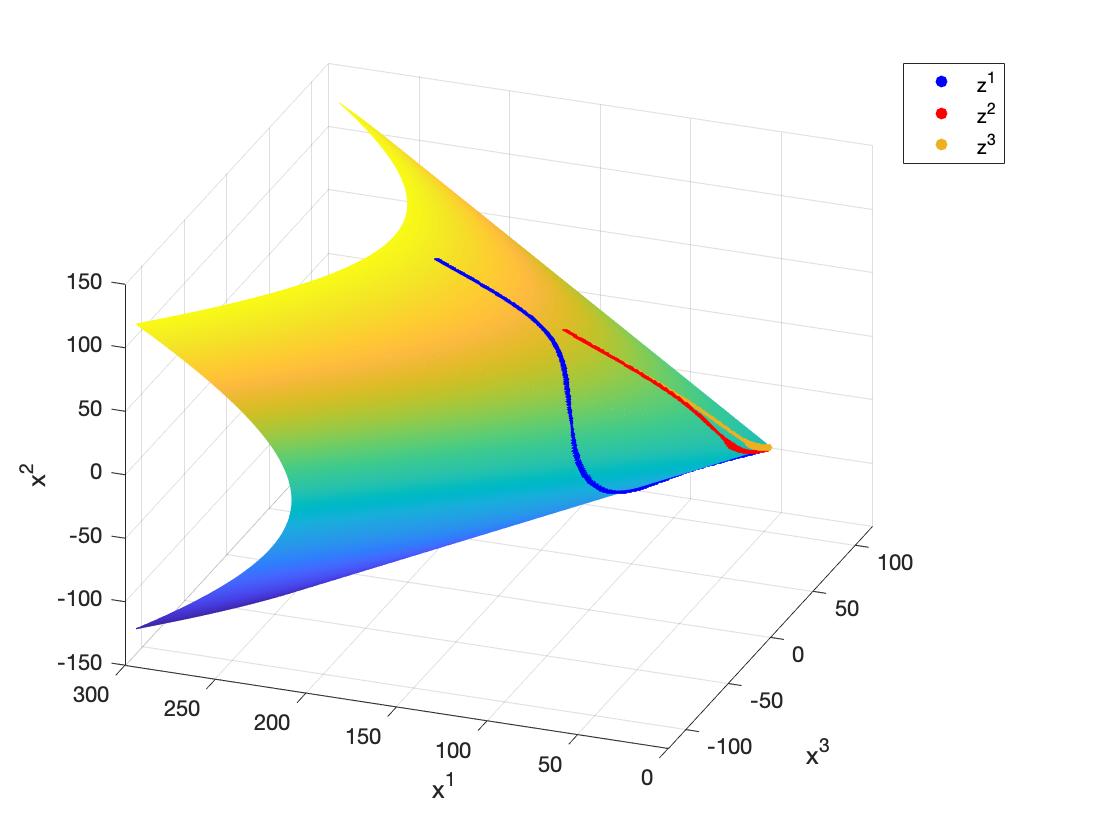}
    \end{minipage}
    }
    \subfigure[ ]{
    \begin{minipage}[b]{.45\linewidth}
    \centering
    \includegraphics[scale=0.19]{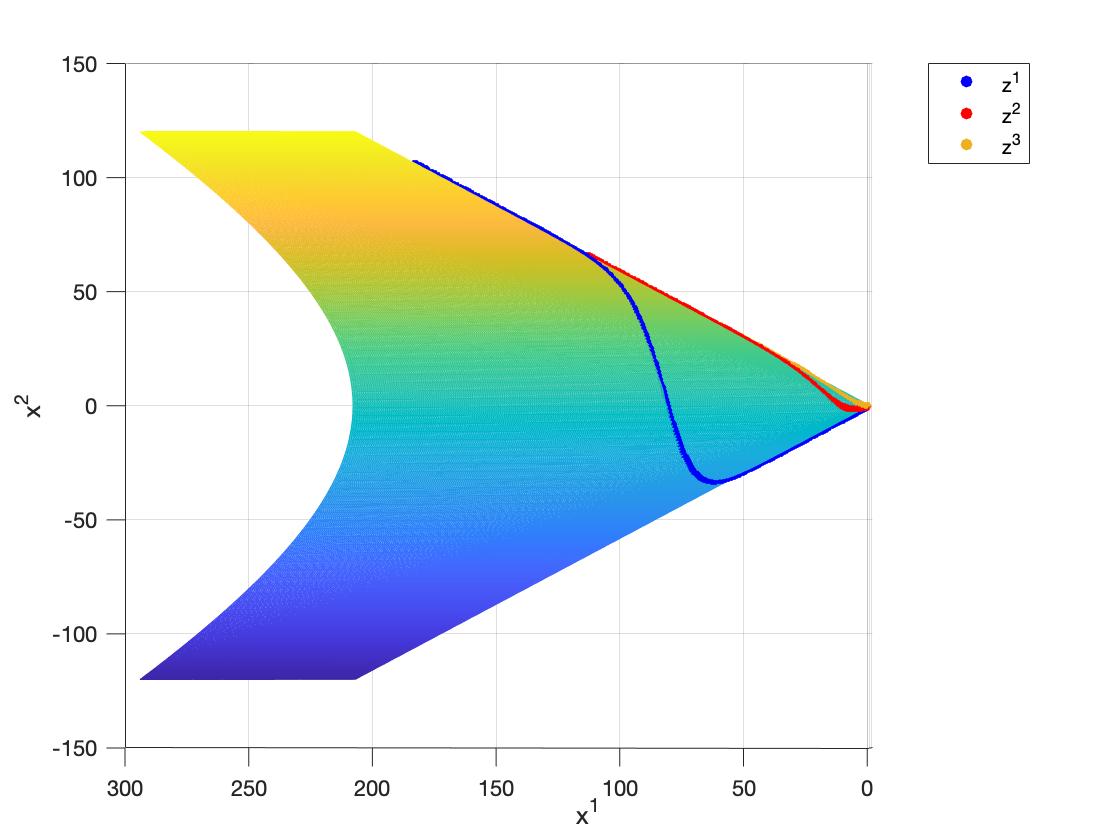}
    \end{minipage}
    }
    \subfigure[ ]{
    \begin{minipage}[b]{.45\linewidth}
    \centering
    \includegraphics[scale=0.19]{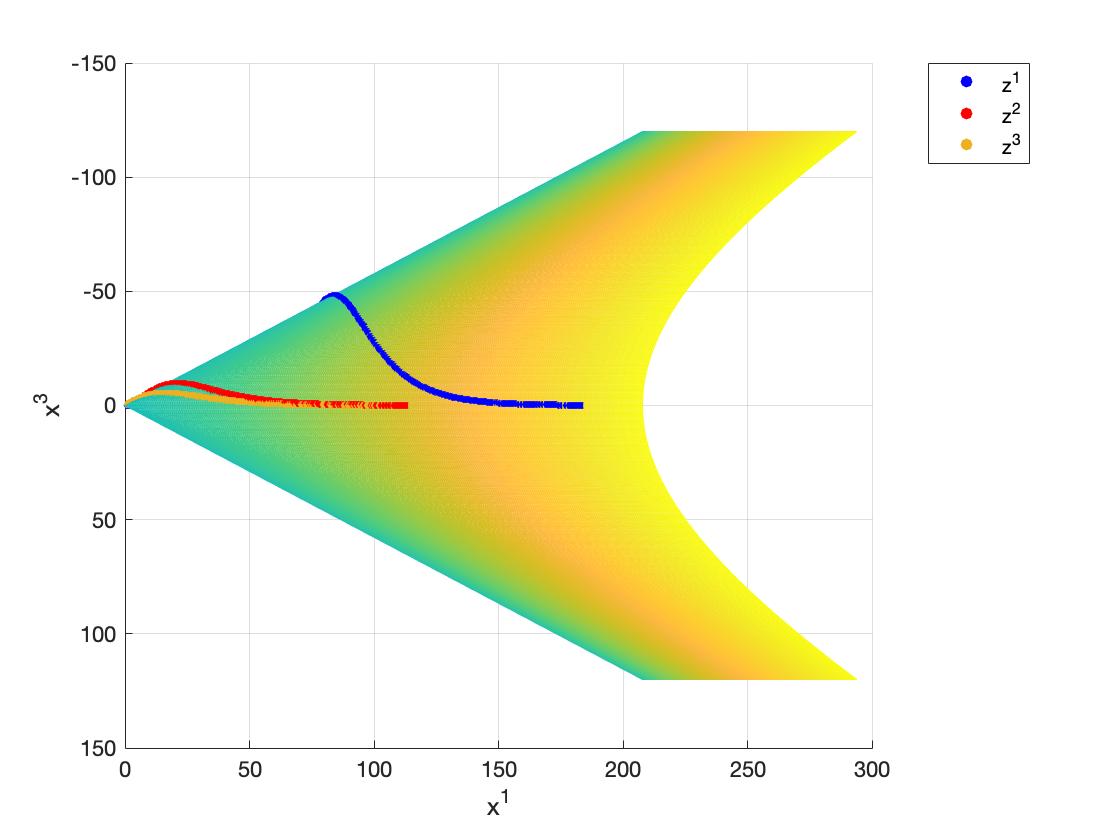}
    \end{minipage}
    }
    \subfigure[ ]{
    \begin{minipage}[b]{.45\linewidth}
    \centering
    \includegraphics[scale=0.19]{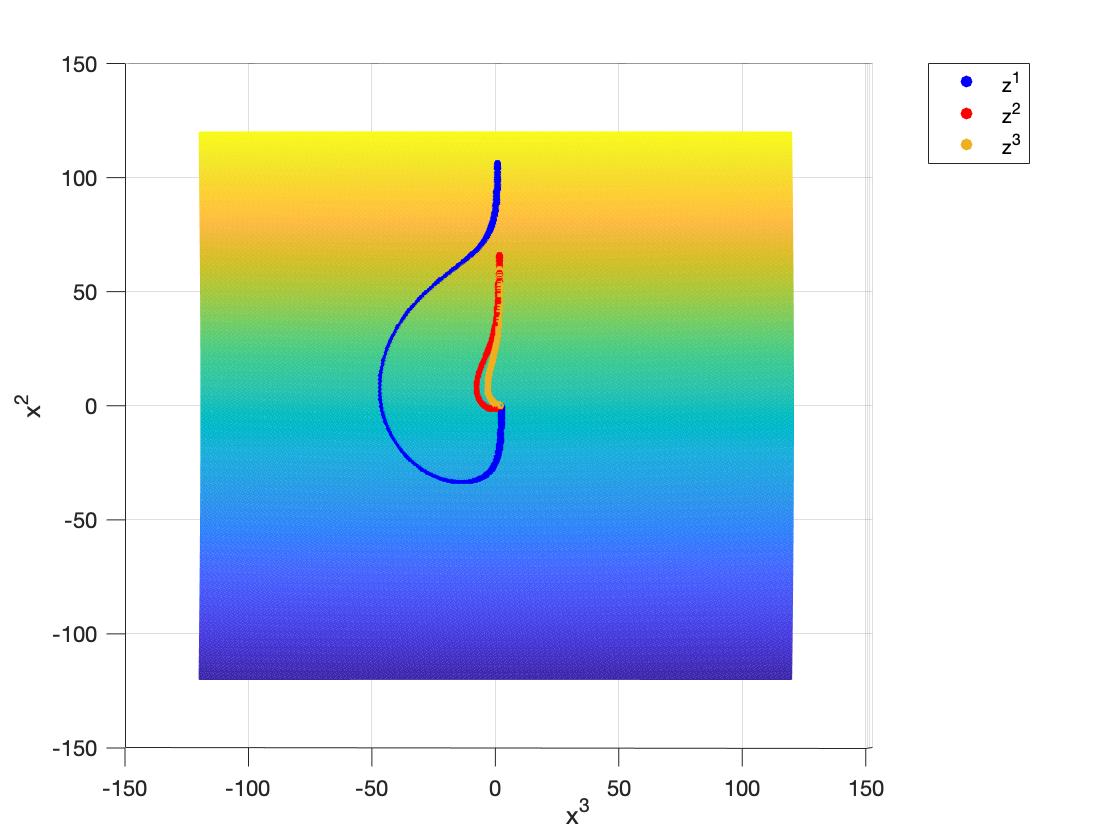}
    \end{minipage}
    }
    
    \caption{Trajectories on the cone with $\theta_0=\frac{\pi}{6},~\alpha_0=\frac{\pi}{36}$.}\label{fig10}
\end{figure}

\section*{Acknowledgements} This work is supported by the National Natural Science Foundation of China under Grants
No. 11871218, No. 12071298, and  in part by Science and Technology Commission of Shanghai Municipality (No. 21JC1402500 and No. 22DZ2229014).


\begin{thebibliography}{3}
        

           \bibitem{4} Anderson, John D. Jr.: \emph{Modern Compressible flow with historical perspective}, Third edition, McGraw-Hill Education, 2003.

        \bibitem{3} Chen, S. and Li, D.: \emph{Conical shock waves in supersonic flow}, Journal of Differential Equations, 269 (2020) 595-611.

           \bibitem{8} Collatz, L.: \emph{The Numerical Treatment of Differential Equations}, Springer-Verlag, Berlin, 1960.

        \bibitem{6} Courant, R. and  Friedrichs, K.O.:~\emph{Supersonic flow and shock waves},~Interscience Publishers,~1948.
            
        \bibitem{5} Coutsias, Evangelos A.,  Hagstrom, Thomas and  Torres, David:  \emph{An efficient spectral method for ordinary differential equations with rational function coefficients}, Math. Comp., 65(1996), 611-635.

        \bibitem{2} Cui, D. and Yin, H.: \emph{Global supersonic conic shock wave for the steady supersonic flow past a cone: polytropic gas,} Journal of Differential Equations, 246 (2) (2009) 641-669.


           
           

        \bibitem{1} Qu, A. and Yuan, H.: \emph{Radon measure solutions for steady compressible Euler equations of hypersonic-limit conical flows and Newton's sine-squared law}, Journal of Differential Equations, 269 (2020) 495-522.
      
        \bibitem{7}  Ruban, Anatoly I. and  Gajjar, Jitesh S. B.: \emph{Fluid Dynamics,~Part 1:~Classical fluid dynamics}, First edition, Oxford University Press, 2014.
     
      
       

    
\end{thebibliography}
\end{document}